\documentclass{article}
\usepackage{arxiv}
\usepackage[pagewise]{lineno}%\linenumbers

\usepackage[url=false,date=year,backend=bibtex]{biblatex}%urldate=false,
\usepackage{tikz}
\usepackage{graphicx} 
\usepackage{subcaption}
\usepackage{multirow}
\usepackage{amsmath, amsthm, amssymb, bbold}
\DeclareMathOperator*{\argmax}{argmax}
\DeclareMathOperator*{\argmin}{argmin}
\usepackage{algorithm}
\usepackage{algpseudocode}
\usepackage{url}
\usepackage{hyperref}
\newtheorem{proposition}{Proposition}

\newcommand{\R}{\mathbb{R}}
\newcommand{\diff}{\:\mathrm{d}}

\usepackage{xcolor}
\usepackage[normalem]{ulem} % for \sout

\newif\ifshowcomments
\showcommentstrue      % change to \showcommentsfalse to hide everything

\author{
David Gentile \\
Tufts University \\
\url{david.gentile@tufts.edu} \\
\And Joshua Huang \\
Tufts University \\
\url{joshua.huang@tufts.edu} \\
\And James M. Murphy\\
Tufts University \\ 
\url{jm.murphy@tufts.edu}}

\date{\today}
\hypersetup{
 pdfauthor={David Gentile},
 pdftitle={An Application of the Wasserstein Metric to the Change Point Detection Problem},
 pdfkeywords={},
 pdfsubject={},
 pdfcreator={Emacs 30.2 (Org mode 9.7.34)}, 
 pdflang={English}}
\begin{document}

\title{Wasserstein-Based Identification of Metastable States in time series Data via Change Point Detection and Segment Clustering}
\maketitle
\begin{abstract}
   Change point detection for time series analysis is a difficult and important problem in applied statistics, for which a variety of approaches have been developed in the past several decades. Here, the Wasserstein metric is employed as a tool for change-point identification in multi-dimensional time series data in order to identify clusters in time series in an unsupervised way.  We leverage the simplicity of the optimal transport cost in the 1-dimensional setting to quickly identify both a segmentation (family of change points for a trajectory) and a clustering for the data when the number of segments is much smaller than the number of data points, making no parametric assumptions about the particular distributions involved. Our change point detection method scales linearly in the size of the data and in the dimension of the samples. We test our approach on idealized synthetic data trajectories, as well as real world trajectories coming from the domain of molecular dynamics simulations and underwater acoustics. We find that segmenting these time series via change points obtained by estimating the Wasserstein metric derivative and then clustering the identified segments as measures with similarity measured by the Wasserstein metric, successfully identifies metastable states in the law of the processes.
\end{abstract}
\section{Introduction}
\label{Introduction}
\subsection{Motivation}
Change point detection (CPD) is the problem of identifying a segmentation of a time series which partitions the data into a small number of internally coherent sequences.  It is an essential task in the analysis of time series data, and many different methods, covering a wide breadth of approaches, exist to tackle the problem \cite{aminikhanghahiSurveyMethodsTime2017}. As a statistical tool, it finds applications in finance \cite{kimUnsupervisedChangePoint2022, habibiBayesianOnlineChange2021, pepelyshevRealtimeFinancialSurveillance2017}, climate modeling \cite{reevesReviewComparisonChangepoint2007, beaulieuChangepointAnalysisTool2012, gallagherChangepointDetectionClimate2013}, speech analysis \cite{tahmasbiChangePointDetection2008, cmejlaBayesianChangepointDetection2013, guptaSpeakerChangePoint2015}, edge detection \cite{tourneretBayesianOfflineDetection2003, huangLaplacianChangePoint2020, huangStatisticalTheoryEdge1988}, and health monitoring \cite{liuChangepointDetectionMethod2018, alhassanBridgingDataDiagnostics2025, wuUnsupervisedBayesianChangepoint2024}.  The task of a such an algorithm is to produce a segmentation of a time series by identifying the specific observations at which a significant qualitative change occurs --- for example, the moments where blood pressure becomes excessively elevated, or when anomalous trades are made in the stock market.  In this work, our primary interest is in leveraging the Wasserstein metric on probability measures to facilitate change point detection and, together with classic techniques in data clustering, to produce compact descriptions of large and complicated time series data sets, specifically trajectories emerging from molecular dynamics simulations (MDS).  This approach is an inversion of a pattern observed in some CPD pipelines \cite{adamsBayesianOnlineChangepoint2007, zhengLearningTransportationMode2008, reddyUsingMobilePhones2010}, wherein clustering is performed first and used to identify segmentation of the data; here we instead use change point detection to facilitate the clustering of the data by treating the segmentation as a reduced representation of the original information.

CPD is a powerful tool for the study of time series data because although a given trajectory may be large in the sense that the number of observations is large, containing potentially millions of samples, it may also be redundant if the time series tends to oscillate between few metastable states wherein the variance between observations is low. Therefore, if both the number of metastable states and the rate at which the trajectory passes between them is relatively small, significant reduction in the number of pairwise comparisons needed to apply a distance based clustering technique can be achieved by viewing it as a sequence of empirical distributions representing segments of the trajectory. How one searches for these change points is therefore a key problem for this approach. %\DCGedit{Broadly, being statistical methods, CPD schemes will either be parametric or non-parametric in the sense that assumptions are made on the underlying forms of the distributions driving the data-generating process. A major challenge associated with this approach is to develop a suitably general change point detection algorithm, one that can be employed in many different settings for molecules with metastable state for which we have no principled statistical priors, and which works quickly enough to tune any hyperparameters which may exist in the algorithm.}{Accidentally marked your comment here as resolved! Whoops. But in any case I cut this sentence and moved it to the lit review at the top of sec. 2, and I've swapped this paragraph with the next so that the general flows into the specific a little more naturally.}

The motivating problem for this paper arises from studying the trajectory data of MDS. A common assumption when studying these trajectories is that the configuration space of a given particle in a fixed setting can be partitioned into regions which represent particle states which are ``metastable'' and transitions between them; the metastable regions of the space are those where the particle tends to remain for extended periods of time, e.g., the basins of an energy well, while the transitions are the complement of the metastable sets. A natural question follows from this assumption, namely, how can we construct reduced-order models of a particle's behavior based on its trajectory data? One approach to the problem is \textit{segment-based clustering}, a technique explored in \cite{damjanovicCATBOSSClusterAnalysis2021}. By looking for change points in the time series data, a very large number of data points is reduced to a relatively small number of segments, leveraging the time sequence aspect of the data to reduce the number of pairwise distances needed to compute a similarity matrix, which is then used to identify a clustering of those segments via the clustering by density peaks algorithm \cite{rodriguezClusteringFastSearch2014, wangDensityPeakClustering2024}. The individual data points then inherit the label of their parent segment.

One of the CPD methods which originally inspired this work is BarT \cite{damjanovicModelingChangesMolecular2023}, which is a transport-based modification of the SIMPLE algorithm \cite{IdentifyingLocalizedChanges}, both of which depend on maximum likelihood estimation \cite{wassermanAllStatisticsConcise2004} and consequently must have a statistical prior about the distributions \cite{IdentifyingLocalizedChanges} \cite{damjanovicModelingChangesMolecular2023}. Both of these are methods that take an optimization-oriented view of the CPD problem, where a loss function is established which is based on the log-likelihood of observing the trajectory for a given set of change points, together with a regularization term which modulates how many change points the algorithms produce by penalizing each change point, with a lower penalty for simultaneously detected changes in several observables. These optimization procedures are extremely expensive, meaning that finding the optimal penalizing parameter for the regularization term is impractical. By changing our statistical view from a parametric to a non-parametric one, we give up theoretical guarantees about optimality of our change points, but gain a method which can analyze a trajectory much more quickly.

Optimal transport provides a toolbox which is particularly well-suited to data which has a spatial aspect. In particular, it allows us to quantify the difference between two probability measures in a way which encodes the underlying geometry of the space on which the measures are supported. This is important if we are trying to discern the points in time at which the trajectory leaves or enters distinct regions of a configuration space for extended periods via observed data. Note also that this is an advantage that transport has over other commonly employed ways of comparing distributions when the supports of the distributions differ --- for example, the Kullback-Liebler divergence is not even defined for two measures with disjoint support \cite{wassermanAllStatisticsConcise2004}. We apply this transport toolbox to the CPD problem by viewing a segmentation of a time series as a sequence of empirical probability measures. Using a moving window, we scan along the time series, measuring the distance between neighboring segments using the optimal transport cost, then identify change points as those times where the distance between the empirical measures representing two neighboring segments is unusually large.

We conceive of our time series as being a realization of some stochastic process with hidden law. If the law of our stochastic process evolves in a continuous way, then we can consider it to be a parameterized curve in the space of measures, i.e., a particle trajectory in something resembling a manifold. If this particle trajectory is concentrated in a small region of the manifold for some period of time, then the stochastic process is in a metastable state. If the particle is moving between metastable states it is in a transition state. After segmenting the trajectory by identifying change points and transitions between metastable states we will produce a reduced-order description of the trajectory by grouping together empirically observed metastable segments, using classical unsupervised clustering methods.

Following our analogy to particle trajectories in a manifold, our heuristic for metastable state detection will be the following: the particle is not moving (and hence metastable) if the speed of the trajectory is relatively small. Our approach is data-driven, and in general we cannot hope to have access to, e.g., an analytic stochastic differential equation governing the trajectory, so we must approximate the differential of the curve based on available data. %\DCGedit{Additionally, while the full Wasserstein gradient (again, thinking of \((\mathcal{P}(\Omega), W_2)\) as a Riemannian manifold) could be employed for this task, is already more information than we need. While the direction of the trajectory may be interesting in its own right, and worth investigating in future work, our simple heuristic actually only depends on the speed of the particle, and hence an estimator of the \textit{metric Wasserstein derivative} will suffice for our purpose of detecting whether or not movement is occurring in the space of probability measures.}{}

To this end, we will apply the \(2\)-Wasserstein metric as a non-parametric statistical tool for the analysis of noisy time series data. Formally, our hypothesis is this: suppose we are observing a stochastic process \(X_t\) on a compact interval \([0, T]\), and that there is some time \(t_0 \in (0,T)\) where a change in the law of the process occurs. Then, if \(\hat{\mu}_{t_{0}^{-}}\) is the empirical distribution of the process generated by the samples \(\left\{ X_t \right\}_{t \le t_0}\) and \(\hat{\mu}_{t_{0}^{+}}\) is the empirical distribution generated by the samples \(\left\{ X_t \right\}_{t_0\ge t}\), i.e.,
\[\hat{\mu}_{t_0^-} = \sum\limits_{t=0}^{t_0}\frac{1}{t_0}   \delta_{X_{t}}, \quad \hat{\mu}_{t_0^+} = \sum\limits_{t=t_0+1}^{T}\frac{1}{T -t_0}\delta_{X_{t}},\]
\newpage
then we expect that \(t_0\) should solve
\begin{equation*}
\argmax\limits_{0\le t \le T} W_2^2(\hat\mu_{t^{-}}, \hat\mu_{t^{+}}).
\end{equation*}
Naturally, we do not expect to see only one change point in a given trajectory --- rather than looking for a global maximizer of the Wasserstein distance, we will examine local maximizers. To find such maximizers, we employ a sliding window approach, fixing a window size \(w\), and compute for each \(t = w, \dots , T - w\) the Wasserstein distance between the empirical measures defined by the previous and the next \(w\) samples. Then, a cut-off is applied to identify outliers where the distance between consecutive windows is unusually large.

Once we have identified a family of change points, we have determined a segmentation of pairwise disjoint contiguous subsets of our trajectory. We would now like to classify the segments in an unsupervised manner. To do this, we once again interpret the segments as empirical measures, whose pairwise similarity can be measured via the \(2\)-Wasserstein metric.
The clustering method we employ is clustering by advanced density peaks (ADPC) \cite{rodriguezClusteringFastSearch2014,wangDensityPeakClustering2024}. ADPC is advantageous because it is entirely unsupervised, determining a number of cluster centers automatically via a statistical criterion.

\begin{figure}[htbp]
\centering
\includegraphics[width=0.5\linewidth]{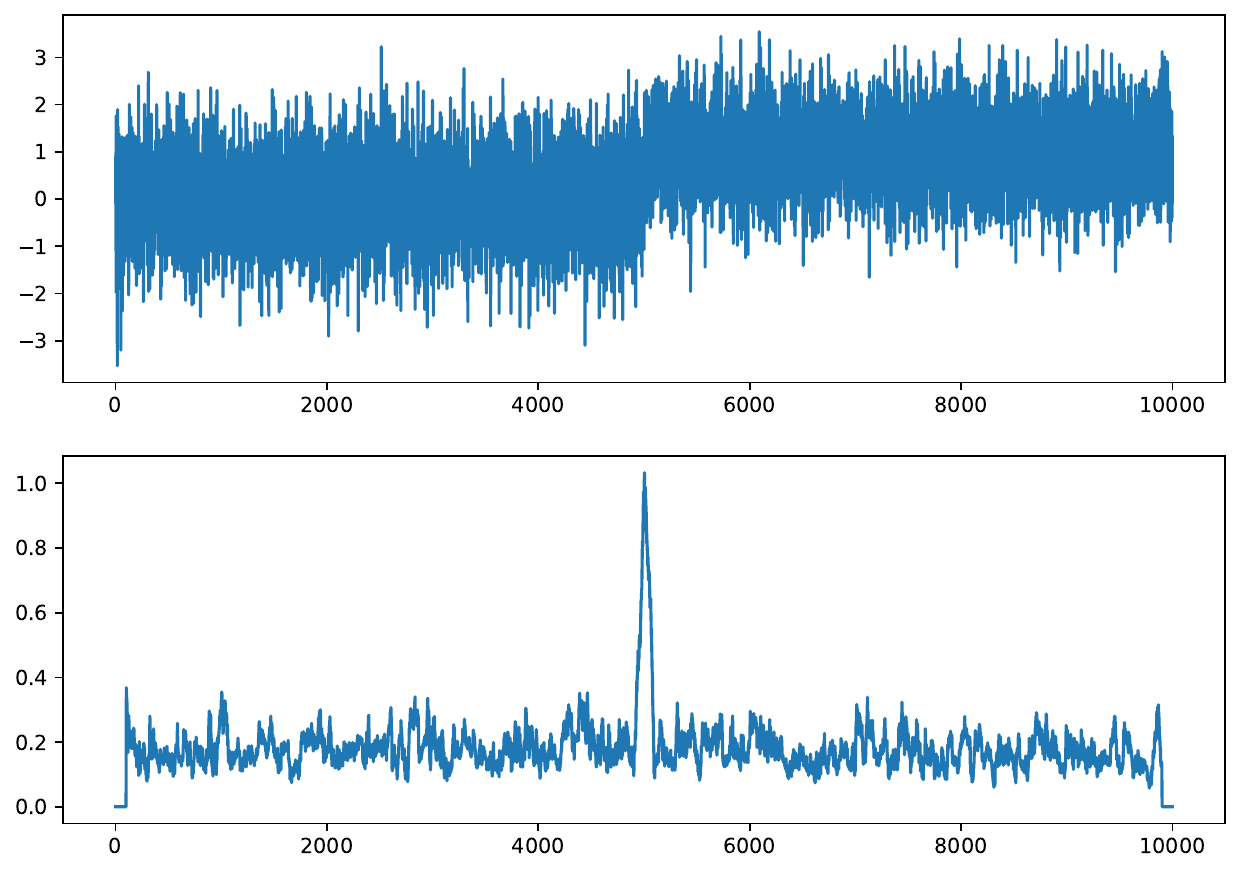}
\caption{The simplest possible instance of a change point: an abrupt change in the law occurs at time t=5000. The change in the law of the process is reflected in the Wasserstein metric derivative. Top: an artificial time series, with first 5000 samples drawn from a \(\mathcal{N}(0, 0.85)\) distribution and the second 5000 samples drawn from a \(\mathcal{N}(1, 0.75)\) distribution. Bottom: approximated metric Wasserstein derivative of the timeseries, with extremum appearing at \(t=5000\), corresponding to the abrupt change in the distribution.}
\end{figure}

\subsection{Contributions}
In this work, we present a non-parametric change-point detection algorithm, and combine it with an unsupervised clustering approach to identify metastable states in time series data. We find that in idealized synthetic data sets, molecular dynamics simulation data, and underwater acoustics data that our approach is able to identify coherent clusters in the space of probability distributions, thereby providing a compact model of the data. Our CPD approach scales linearly in the size of the data, and is thus well-adapted to large-scale data sets, and, provided that the data does not oscillate between states too rapidly, has the advantage of greatly reducing the number of pairwise distances that need to be computed for the cluster identification. 

\subsection{Organization of the paper}
In Section 2, we outline the necessary background to understand both the change point detection method and the clustering algorithms employed. In Section 3, we describe the algorithm we used to obtain our numerical results. In Section 4, we present the numerical results of our experiments, demonstrating efficacy of our approach in a variety easily understood low-dimensional examples. In Section 5, we conclude with a discussion of our results, and propose some future directions for expanding and extending the work presented herein.

\section{Background}
\subsection{Previous Work on Moving Window Change Point Detection and Relation to the Literature}
CPD encompasses a vast collection of methods from the field of applied statistics and for a comprehensive overview of the field, we refer the reader to the  surveys \cite{aminikhanghahiSurveyMethodsTime2017, reevesReviewComparisonChangepoint2007, niuMultipleChangePointDetection2016, truongSelectiveReviewOffline2020} as well as the textbook \cite{BassevilleNikiforovBook}. Generically speaking, CPD methods have several qualities which are used to categorize them, which we outline below. 

CPD methods may be online or offline. Online methods are those which run concurrently with the process about which they are trying to make inferences, whereas offline methods take an after-the-fact approach to the problem of identifying changes in the trajectory, examining all the data at once \cite{pageCONTINUOUSINSPECTIONSCHEMES1954,picardTestingEstimatingChangepoints1985}.  Celebrated online approaches include \cite{adamsBayesianOnlineChangepoint2007, jacksonAlgorithmOptimalPartitioning2005, OnlineAlgorithmSegmenting}, and have seen a rise in popularity as hardware improvements make real-time change detection more feasible. Our method is an offline approach, because we use statistical quantities based on the entire trajectory, retrospectively segmenting the time series and then looking for patterns in the data.

Another distinction to make when differentiating CPD methods is whether or not they are driven by some parametric assumption on the underlying data. Parametric approaches to the CPD problem are among some of the earliest, dating back to the Wald's sequential probability ratio test \cite{waldSequentialTestsStatistical1945}. Parametric and non-parametric approaches offer distinct advantages and trade-offs; the former work by making an explicit assumption on the underlying nature of the distribution (e.g., a process is stationary and Gaussian), while models in the latter category make no such assumptions. 

Parametric models may incur higher computational costs during the process of optimizing model selection. For example, SIMPLE \cite{IdentifyingLocalizedChanges}, one of the algorithms that served as initial inspiration for our work, is a parametric CPD scheme which optimizes a loss function which must be derived on some statistical prior, e.g. the data is drawn from a family of distributions which is Gaussian, Laplace, exponential, and so on. Non-parametric approaches by contrast avoid model selection issues, and but tend to require access to larger amounts of data at inference time \cite{aminikhanghahiSurveyMethodsTime2017}. Tools previously used in non-parametric CPD include but are not limited to the Kolmogorov-Smirnov test \cite{picardTestingEstimatingChangepoints1985, padillaOptimalNonparametricChange2019, patistOptimalWindowChange2007a}, graphs \cite{chenGraphbasedChangepointDetection2015, chenGraphBasedChangePointAnalysis2023}, and, of particular interest to us, optimal transport, where approaches based on distribution-free methods \cite{chengOptimalTransportBasedModels2023, chengMatchedFilteringStatistical2020, faberWATCHWassersteinChange2021}, kernel-based methods \cite{wangTwoSampleTestKernel2022}, and entropic regularization of the Wasserstein distance have been studied  \cite{werenskiRankEnergyStatistics2024}. Our approach is similar to the one that appears in \cite{chengOptimalTransportBased2020, chengOptimalTransportBasedModels2023}, with some key differences --- in particular, our use of the metric derivative rather than the Wasserstein two-sample test.  The Wasserstein two-sample test \cite{delbarrioTestsGoodnessFit1999, ramdasWassersteinTwoSampleTesting2017} is a distribution free non-parametric test which is applicable to 1D distributions on the real line and measures whether or not the monotone rearrangement of one distribution into another differs from the identify function in a statistically significant way; if there is no such significant difference, the null hypothesis (both samples are drawn from the same distribution) is accepted. This differs our approach, which we will describe below, in that we work directly with the Wasserstein distances between neighboring distributions and look for outliers in that data.  Additionally, we are interested in data sets which take values in a non-Euclidean domains, such as the \(3\)-torus. Recently Wasserstein two-sample tests have been studied on the flat torus \cite{gonzalez-delgadoTwosampleGoodnessoffitTests2023}, and this may provide an avenue for future extension of our work.

As a last categorical distinction among approaches, CPD methods are either supervised or unsupervised. The regime of the former is one in which change points are labeled ahead of time, and the goal of training is to develop a model which predicts well labels excluded from the training data. On the other hand, unsupervised CPD aims to identify structures in the data without knowledge of a ground truth --- this is the regime under which our method operates.  One common iteration of the unsupervised approach is to cluster the data into a family of states, and look for the points where the cluster label changes. Classic examples of this approach include the sliding window and bottom-up (SWAB) algorithm \cite{keoghOnlineAlgorithmSegmenting2001} and the minimum description length (MDL) algorithm \cite{rakthanmanonTimeSeriesEpenthesis2011}. Our approach reverses this paradigm, and attempts to identify suitable labels for the data based on the identified change points.

One challenge of moving window schemes, and which is acute to our approach in particular, is the choice of the window size \(w\). Being an essentially statistical estimator of the instantaneous distribution, it is naturally preferable for our windows to have a large number of samples, and thus be better estimators of the true distribution. On the other hand, as a practical matter, it isn't hard to take \(w\) too large --- supposing for a moment that there were some set of true change points, if the average (true) segment length is much smaller than the window size, then a typical window will contain many segments, and thus our approach will have no hope of detecting the change points in the large \(w\) regime. Therefore, our algorithm leaves \(w\) as a tunable parameter, which the user can adjust as fits their use case.

The primary predecessors which inspired our work are the SIMPLE and BarT algorithms, both of which were developed for specific application to molecular dynamics trajectory data \cite{damjanovicModelingChangesMolecular2023}, with the latter being an extension of the former based on optimal transport and the Wasserstein distance. SIMPLE works for a process \(\{\{Y_{j,t}\}_{t=1}^T\}_{j=1}^J\)in \(J\) dimensions by first choosing a log-likelihood function \(\hat{l}\) representing the plausibility of observing the given data conditioning on a set of change points \(\{\tau\}_{i=1}^{K}\), and proceeds by minimizing by means of a dynamic programming algorithm 
\begin{equation*}
    \sum\limits_{j=1}^J\sum\limits_{i=0}^K \hat{l}(Y_{j,\tau_{j,i}+1}, \dots, Y_{j,\tau_{j,i+1}}) - \lambda\sum\limits_{i=1}^K q(S_i)
\end{equation*}
where \(S_i\) is a subset of \(\{1, \dots, J\}\) corresponding to the dimensions which are changing at time \(\tau_i\), and \(q:2^{\{1,\dots,J\}}\to\R\) is a strictly increasing penalty function satisfying \(q(A \cup B) \le q(A) + q(B)\) whenever \(A\) and \(B\) are non-empty disjoint subsets, e.g. in the original implementation of SIMPLE, the authors take \(q(\cdot) = |\cdot|^\alpha\) for \(\alpha \in (0,1]\). This encourages the algorithm to choose change points where many dimensions are changing simultaneously --- any change point whatsoever incurs a penalty, but points which exhibit changes only in a few dimensions of the data set are penalized more heavily.

BarT takes a point of view based on optimal transport (see Section \ref{otsec} for details on optimal transport), adapting the loss function \(\hat{l}\) to consider the possibility that a point belongs to a barycentric transition between metastable states. This adaptation is made by taking \(\hat{l}(\cdot) = \max\{L_{ms}(\cdot), L_{tr}(\cdot)\)\}, where \(L_{ms}\) and \(L_{tr}\) are, respectively, the log-likelihood functions corresponding to a metastable and a transition region. Explicitly, the algorithm supposes that metastable states are Laplace distributed, and takes \((\mu_j, \sigma_j)\) to be the parameters of the \(j\)-th segment \(\{X_t\}_{t={\tau_j}}^{\tau_{j+1}}\). Then one has
\begin{equation*}
    L_{ms}(\{X_t\}_{t={\tau_j}}^{\tau_{j+1}}) = \log\prod\limits_{t={\tau_j}}^{\tau_{j+1}} \frac{1}{2\sigma_j}\exp\left(-\frac{|X_t - \mu_j|}{\sigma_j}\right), \quad L_{tr}(\{X_t\}_{t={\tau_j}}^{\tau_{j+1}}) = \log\prod\limits_{t={\tau_j}}^{\tau_{j+1}} \frac{1}{2\sigma_{j,t}}\exp\left(-\frac{|X_t - \mu_{j,t}|}{\sigma_{j,t}}\right),
\end{equation*}
wherein \((\mu_{j,t}, \sigma_{j,t})\) correspond to the parameters of the Wasserstein geodesic between the \((j-1)\)-th and \((j+1)\)-th segments. These parameters are easily computed because the Wasserstein geodesic of two Laplace distributions \(\gamma_1\) and \(\gamma_2\) is itself Laplace, and the parameters of the (unit time length) interpolant are directly estimable from the estimated parameters of \(\gamma_1\) and \(\gamma_2\) \cite{damjanovicModelingChangesMolecular2023}. 
Our work is inspired by this transport-driven approach, but differs essentially in that it forsakes an explicit loss function, which is expensive to optimize, in favor of working with the geometric structure that the Wasserstein metric gives to the set of empirical distributions. 

The problem of change point detection and segment clustering for time series data was also considered in \cite{chengOptimalTransportBased2020}, although our work takes a somewhat different approach to the problem for 1D timeseries, and further expands the approach to multi-dimensional non-Euclidean timeseries data, both those taking on vector values and those satisfying a manifold hypothesis e.g., we will apply our methods to molecular dynamics simulation data which can be represented as a timeseries on the torus.

\subsection{Optimal Transport as a Geometry on Probability Distributions}\label{otsec}

Optimal transport is broad field, and for a comprehensive overview of the subject we refer the reader to any of \cite{santambrogioOptimalTransportApplied2015, ambrosioLecturesOptimalTransport2021, villaniOptimalTransport2009, peyreComputationalOptimalTransport2020}. What follows is an abbreviated summary of the most necessary details.  In the sequel, we follow the presentation found in \cite{santambrogioOptimalTransportApplied2015}.

Let \((\Omega, d)\) be a complete, separable metric space; typically, one takes \(\Omega = \R^d\) with the standard Euclidean norm. Let \(\mathcal{P}_2(\Omega)\) be the space of finite-variance probability measures on \(\Omega\), let \(\mu, \nu \in \mathcal{P}_2(\Omega)\) with \(c : \Omega \times \Omega \to \mathbb{R}\), and finally let \(\Pi(\mu,\nu)\) be the set of probability distributions on \(\Omega\times\Omega\) with marginals \(\mu\) and \(\nu\). Then the \textit{optimal transport cost} for \(\mu\) and \(\nu\) with cost \(c\) is given by 

\begin{equation*}
    OT(\mu,\nu) := \inf\left\{\int_{\Omega\times\Omega} c(x,y) \: \diff \gamma(x,y) \: : \: \gamma \in \Pi(\mu, \nu)\right\}.
\end{equation*}If one takes \(c(x,y) = d^2(x,y)\), then the square root of the optimal transport cost defines a metric on \(\mathcal{P}_2(\Omega)\), which is known as the \(2\)-Wasserstein metric, and is denoted by \(W_2\). Furthermore, it is in this setting that one has the formal interpretation of \((\mathcal{P}_2(\Omega), W_2)\) as a Riemannian manifold \cite{jordanVariationalFormulationFokker1998,benamouComputationalFluidMechanics2000}. The manifold interpretation of the space of probability distributions endows the space with a dynamics, which is essential to the motivation of our approach.
In the case where measures are supported on the real line and absolutely continuous with respect to the Lebesgue measure, the following convenient identity holds:

\begin{equation}
\label{ot1d}
    W_2^2(\mu, \nu) = \int_0^1 |F_\mu^-(x) - F_\nu^-(x)|^2 \: \diff x,
\end{equation}where \(F_\mu^{-1}\) is the generalized inverse cumulative distribution function, defined by

\begin{equation*}
    F_\mu^{-1}(x) = \inf \{t : \mu([-\infty, t) \ge x\}.
\end{equation*}

This reason this formula is so convenient when working with one-dimensional empirical measures is that it reduces the problem of computing an optimal transport cost to the problem of sorting two lists, which is a much easier problem than solving a general linear program.

While a notion of ``gradient'' exists for the Wasserstein geometry \cite{santambrogioOptimalTransportApplied2015}, for our purposes it will suffice to work with the Wasserstein \textit{metric derivative} \cite{ambrosioGradientFlowsMetric2005}, defined as follows: let \(\gamma:[0,1]\to\mathcal{P}_2(\Omega)\), let \(t_0 \in (0,1)\). We say that \(\gamma\) has metric derivative \(|\dot\gamma|\) at \(t_0\) if the following limit exists:

\begin{equation}
\label{metric_deriv}
    |\dot\gamma| = \lim\limits_{t\to t_0}\frac{W_2(\gamma(t), \gamma(t_0))}{t - t_0}.
\end{equation}

Metric derivatives may be defined in any metric space for which a notion of absolute continuity is available.  In \cite{ambrosioMetricSpaceValued1990}, it was shown that the limit in \eqref{metric_deriv} exists under very general conditions on the measures \(\gamma(t)\), namely that \(\gamma(t)\) be absolutely continuous in time with respect to the Lebesgue measure on \([0,1]\) for almost-all \(t\). 

The intuition for the metric derivative is that it represents the instantaneous speed of the law of the process in the space of probability distributions. The metric derivative is the starting point for the much-celebrated theory of gradient flows on metric spaces \cite{ambrosioGradientFlowsMetric2005}, which could in the future serve as a good basis from which to expand this work.
\begin{figure}
    \centering
    \includegraphics[]{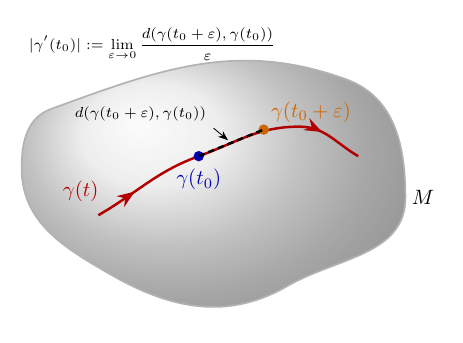}
    \caption{A curve on a manifold; the metric derivative is defined in terms of a limit of a difference quotient where the numerator is given by the difference between nearby points.}
\end{figure}

The Wasserstein distance metrizes the space of probability distributions in a way which encodes the geometry of the underlying space, making it an apt choice for the analysis of physical trajectories, and is amenable to working with Dirac clouds as estimates of the relevant measures. This allows us to quantify the difference between two neighboring windows of points and approximate the metric derivative of the trajectory in Wasserstein space. This information allows us to predict whether or not the law of the process is metastable or in the midst of transition between metastable states.

At a high level, what is essential to our work is the following. Let \(\{X_t\}_{t=0}^T\) be a stochastic process defined on some metric space \(\Omega\); in practice we consider our data as being a realization of the process \cite{oksendalStochasticDifferentialEquations2010}, and from the given data we will try to draw inferences about the particular realization. For each \(t \in [0, T]\), we can denote by \(\gamma_t\) the distribution of the random variable \(X_t\). We will assume that this parameterized family of measures lives in \(\mathcal{P}_2(\Omega)\)---the set of probability measures on \(\Omega\) with finite second moment ---then we can consider \(\mathcal{P}_2(\Omega)\) as an infinite-dimensional Riemannian manifold, at least in a formal sense by interpreting the Wasserstein metric as coming from a Riemannian structure. If the law of our stochastic process \(X_t\) evolves in a continuous way, then we can consider the parameterized family of measures \(\gamma_t\) to be a curve, i.e. a particle trajectory, in this manifold. If this particle trajectory is concentrated in a small region of the manifold for some period of time, then the stochastic process \(X_t\) is in a \textit{metastable state}. If the particle is moving between metastable states it is in a \textit{transition state}. \textit{A priori}, we generally do not know what the metastable states of a process will be --- our goal is thus to leverage the metric structure of \(\mathcal{P}_2(\Omega)\), given by the \(2\)-Wasserstein distance, to identify stretches of time (hereafter, ``segments'') when the particle is metastable. Then, using classical clustering methods, we will produce a reduced-order description of the trajectory by grouping together empirically observed metastable segments. 

\subsection{Clustering Techniques Based on Similarity Matrices}
Once we have split a trajectory into a sequence of segments, our next task is to group them together in a coherent way. In other words, we want to identify clusters in the reduced data set. Each segment is represented by an empirical distribution, and continuing with the theme established by our CPD method, we will quantify the difference between segments with the \(2\)-Wasserstein metric. Because each element in the Wasserstein space is a measure, there is a restriction on the family of clustering methods which are suited to the task of labeling the segments --- for example, while a Wasserstein based \(K\)-means approach has been investigated in the literature, theoretical results on its scaling remain in development and empirical work suggests the scaling may be poor \cite{zhuangWasserstein$K$meansClustering2022}. To deal with this difficulty, we will employ a clustering technique which works just from the pairwise distances of the segment data. Conversely, working with the pairwise distance matrix further highlights one of the reasons why it might be desirable to segment a time series. \textit{In toto}, our data sets may contain millions of individual data points, making the computation of the pairwise distance matrix prohibitive at the level of points --- however, if the data can be broken into a number of segments which is much smaller than the number of raw data points, that computational burden is eased.

\begin{figure}
    \centering
    \includegraphics[width=0.9\linewidth]{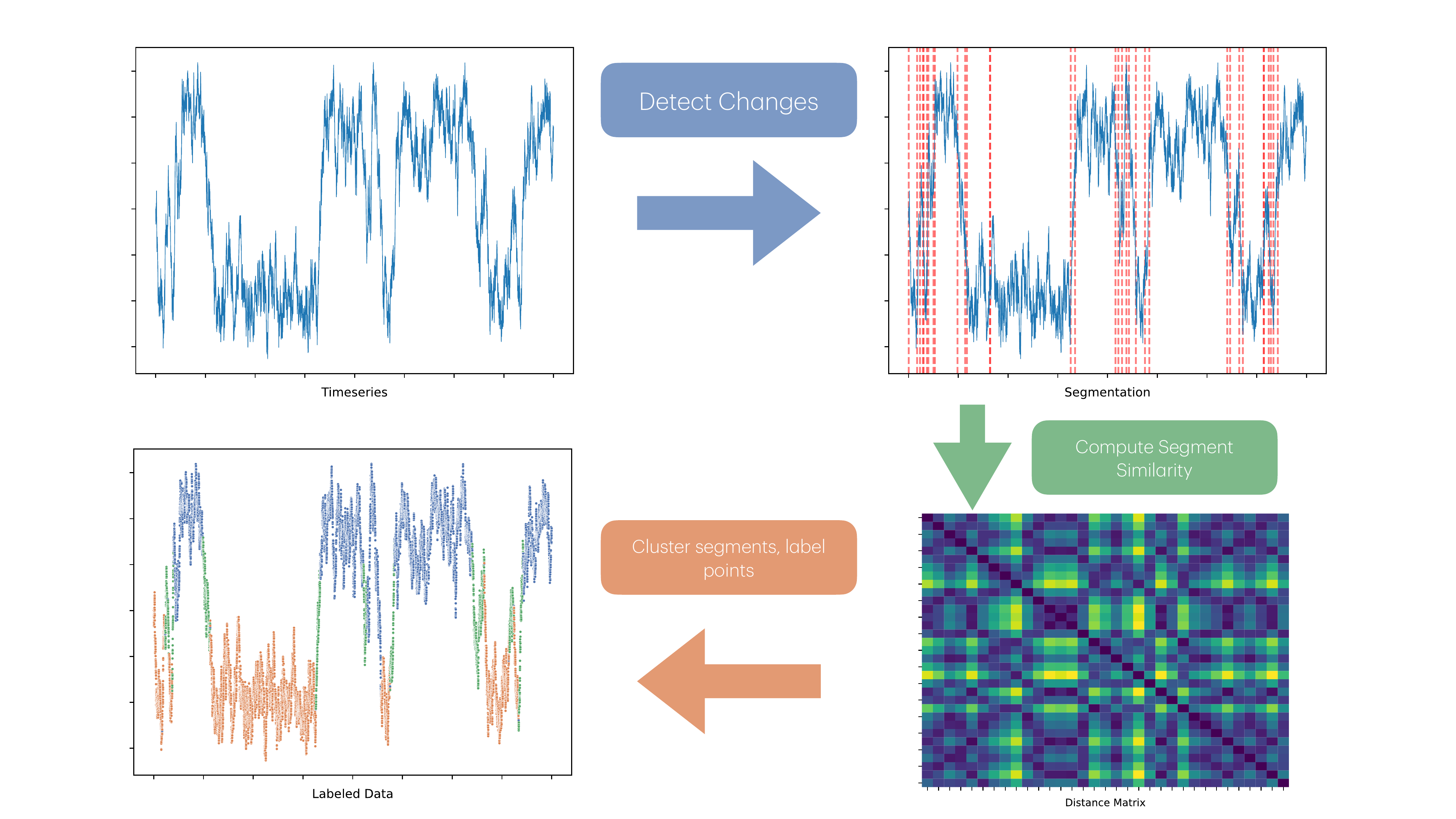}
    \caption{Pipeline of the procedure: first, the time series is split into a family of non-overlapping segments, then a pairwise similarity matrix is computed based on the Wasserstein distance, and finally a clustering algorithm is applied to the similarity matrix to learn labels for the segments and by extension the points of the timeseries.}
\end{figure}

Letting \(S\) be the collection of segments in the time series and \(\hat{\nu_i}\) be the empirical measure associated to each segment, we form the similarity matrix via the rule

\begin{equation*}
    A = \left[\exp\left(-\frac{1}{2\sigma}W_2^2(\hat{\nu_i}, \hat{\nu_j})\right)\right]_{i,j=1}^{|S|}
\end{equation*}
and thus just the formation of the similarity matrix requires \(\mathcal{O}(|S|^2)\) operations.  In this work, we always take \(\sigma=1\). 

%Spectral clustering works by interpreting the similarity matrix as a weighted adjacency matrix for a graph and analyzing the spectrum of the resulting graph Laplacian\cite{stankovicGraphSignalProcessing2019a}. Generally speaking, one of the advantages of spectral clustering is that it is able to identify non-convex clusters --- in other words, if some portion of the data lives on a submanifold of the ambient embedding space which is non-convex with respect to the ambient geometry, spectral clustering can succeed in grouping that portion together, whereas other centroid-based clustering schemes, such as \(K\)-means, will fail to do so.  

Clustering by advanced density peaks is a method that aims to identify the topography of a multimodal dataset \cite{rodriguezClusteringFastSearch2014, derricoAutomaticTopographyHighdimensional2021}. This scheme is driven by several heuristics, and has the advantage of being entirely unsupervised. Firstly, ``putative centers" of the data set are identified by looking for those points such that the local density of points around the centers is maximal. Once a family of centers is identified, this induces a Voronoi partitioning of the ambient space. The borders of this partition are then examined to identify saddle points in the topography of the data set, which in turn allows for the identification of peaks (modes) in the data. All of these steps are driven via statistical methods, i.e. suitable estimators on quantities such as local density. The two main advantages of using ADPC in our setting is that it is both dimension agnostic, in the sense that it will attempt to estimate the intrinsic dimension of the collection of segments implied by the distance matrix on its own, and that it does not require as an input a guess on the number of clusters present in the data. This makes our pipeline amenable to complicated trajectories which cannot be ``eye-balled'', either because the data is essentially high-dimensional or because the modes of the clusters are too close together to be easily differentiated by human judgment.

\section{Algorithms}

It is well-known that in the case of 1D empirical measures on the real line, the optimal transport cost is straightforward to compute, and equivalent to sorting two lists. Following \cite{peyreComputationalOptimalTransport2020}, for two empirical measures \(\hat\mu = \frac{1}{n}\sum\limits_{i=1}^n\delta_{x_i}\) and \(\hat\nu=\frac{1}{n}\sum\limits_{i=1}^n \delta_{y_i}\), which we assume to be sorted so that \(x_i \le x_{i+1}\), \(y_i\le y_{i+1}\) for each \(i\), we have that

\begin{equation*}
    W_2^2(\hat\mu,\hat\nu) = \sum\limits_{i=1}^n |x_i - y_i|^2.
\end{equation*}

The case when the measures have a different number of atoms is not meaningfully more difficult using \eqref{ot1d}, since the psuedo-inverse of an empirical CDF is easily obtained by sorting the observed samples. In our case, we will also be interested in working with 1D transport costs on the unit circle, details for which can be found in \cite{delonFastTransportOptimization2010}. All transport algorithms used for this work are those implemented in \cite{flamaryPOTPythonOptimal2021}.

Beyond the computational aspects of 1D transport costs, our CPD also depends upon two hyperparameters, \(w\) and \(q\). The window size \(w\) specifies the sample size of the empirical distributions we use at each point in time to approximate the law of the process, and \(q\) is a real number between 0 and 1 that specifies a quantile cutoff with respect to the distribution of approximated Wasserstein metric derivatives, indicating that a data point may potentially be labeled as a change point.  In essence, \(w\) controls the number of samples we demand at any given time to approximate the instantaneous change in distribution of the time series, while \(q\) controls how sensitive the algorithm is to the changes in the time series. For example, if \(q\) is taken to be small, then more points will register as candidates, and conversely if \(q\) is taken close to 1, few candidates will be identified.

Our approach is this: for a fixed data set \(\left\{X_t\right\}_{t=1}^T\) and parameters \(w, q\) we approximate the metric Wasserstein derivative via

\begin{equation*}
    \left|\hat{\dot{\gamma}}\right| = W_2(\mu_t^-, \mu_t^+),
\end{equation*}
where \(\mu_t^- = \left\{X_t\right\}_{t-w}^t\) \(\mu_t^+ = \left\{X_t\right\}_{t}^{t+w}\). This approximation gives us a sense of the speed at which the law of the process is moving through the space of probability distributions. Then, with that information in hand, we look for the regions where we see outlier acceleration of the distribution. 

To give a better sense of what exactly it is that we are looking for, consider the ideal scenario in a physical sense: suppose that we are looking at a particle trajectory in the plane over a time interval \([0,1]\). At some time \(t_0\), the particle begins to move uniformly from its initial position to its terminal position, which it reaches at time \(t_1\), after which it ceases to move. The speed of the particle would then be graphically represented by a square wave, i.e, the characteristic function \(\mathbb{1}_{[t_0,t_1)}(t)\). 

\begin{algorithm}
\label{CPDAlg}
\caption{Change Point Detection Algorithm (ComputeChangePoints)}
\begin{algorithmic}[1]
\Require Time series data $X_t$, total number of data points $T$, window size $w$, quantile threshold $q$
\Ensure Set of change point times

\State \textbf{Step 1:} \textit{Compute derivative magnitude array}
\State Initialize $\left|\hat{\dot{\gamma}}\right| \gets $ empty array of size $(T-2w)$
\For{$t = w$ \textbf{to} $T-w$}
    \State Compute $|\hat{\dot{\gamma}}|[t] \gets $ \text{estimate of } $|\dot{\gamma}(t)|$ from $X_t$
\EndFor

\State
\State \textbf{Step 2:} \textit{Identify candidate times based on quantile threshold}
\State Compute $\tau_q \gets$ $q$-th quantile of $|\hat{\dot{\gamma}}|$
\State $\mathcal{C} \gets \{t : |\hat{\dot{\gamma}}|[t] \geq \tau_q, \, t=w, w+1, \dots, T-w\}$ \Comment{Candidate set}

\State
\State \textbf{Step 3:} \textit{Filter for inflection points in contiguous segments}
\State Partition $\mathcal{C}$ into contiguous subsequences $\mathcal{S}_1, \mathcal{S}_2, \ldots, \mathcal{S}_k$
\State Initialize $\mathcal{I} \gets \emptyset$ \Comment{Inflection point set}
\For{each segment $\mathcal{S}_i$}
    \State Identify inflection points within segment $\mathcal{S}_i$
    \State Add detected inflection points to $\mathcal{I}$
\EndFor

\State \Return $\mathcal{I}$
\end{algorithmic}
\end{algorithm}

Implicit in our overall change point detection scheme is a filtering choice: \textit{a priori}, if we set a quantile cut off \(q\), then simply returning the points above that threshold will return \((1-q)\%\) of the data set as change point candidates, and many of these candidate points will be contiguous with each other, as a consequence of the continuity of the trajectory. In other words, we will get out of the first two steps of the scheme a collection of change \textit{regions}, from which we now want to extract a fine-grained selection of change \textit{points}. 

Finally, we extend our approach to multi-dimensional data simply by working component-wise. By this we mean that we treat each each component of the time series as being a process all its own, and perform for each dimension the segmentation and clustering procedure. Formally, we can consider our pipeline as a transformation \(F=L\circ S\) where \(S\) maps the input time series \(\{\vec{X}_t\} \subset \R^D\) to a corresponding set of segments \(\mathcal{S} = \{S_i\}_{i=1}^k\), and where \(L: S_i \mapsto l_i\) maps segments to labels, so that  \(F: \{\{X_{t,d}\}_{t=1}^T\}_{d=1}^D \mapsto \{\{l_{t,d}\}_{t=1}^T\}_{d=1}^D\). Then an overall state is assigned to \(\vec{X}_{t_0}\) by concatenating the labels \((l_{t_0,1}l_{t_0, 2}\dots l_{t_0,d})\).

\begin{algorithm}
\label{IPIDAlg}
\caption{Change Point Detection Subroutine: Identifying Inflection Points}
\begin{algorithmic}[2]
\Require Metric derivative data $\{D_{t}\}_{t=1}^T$, quantile threshold $q$
\Ensure Inflection points of metric derivative data

\State \textbf{Step 1: } \textit{Identify continuous subsequences in \(|\hat{\dot{\gamma}}|\)}

\State \(\mathcal{C} \gets \{t : D_t > \textbf{quantile}(D,q)\}\)
\State \(\mathcal{S} \gets \{\{D_{t}\}_{t=t_i}^{t_{i + 1}} : \{t_i, t_{i}+1, \dots, t_{i} + k\}\subset \mathcal{C} \text{ is a maximum length subsequence}\}\)

\State \textbf{Step 2: } \textit{Report extrema of the subsequences}
\State \(\mathcal{I} \gets \{\}\)
\For{each subsequence \(S \in \mathcal{S}\)}
    %\State Add \(\argmin S_i, \argmax S_i\) to \(\mathcal{I}\)
    \State \(\mathcal{I} \gets \mathcal{I} \cup \{\argmin S, \argmax S\}\)
\EndFor

\State \Return \(\mathcal{I}\)
\end{algorithmic}
\end{algorithm}

\begin{algorithm}
\label{GeneralStepAlg}
\caption{Metastable State Identification Procedure (IdentifyStates)}
\begin{algorithmic}[3]
\Require time series data \(\{\{X_{t,i}\}_{t=1}^{T}\}_{i=1}^d\), window size $w \in \mathbb{N}$, quantile threshold $q \in (0,1)$, clustering method Cluster
\Ensure Point labels \(\{l_t\}_{t=1}^{T}\)
\For{\(i=1:D\)}
\State \textbf{Step 1: } \textit{Identify segmentation of time series}
\State \(\mathcal{C} \gets \text{ComputeChangePoints}(X_t, w, q)\)\Comment{Identify change points}
\State \(\mathcal{S} \gets \{\{X_t\}_{t=t_i}^{t_{i+1}} : t_i \in \mathcal{C}\}\)\Comment{Partition data into segments}
\State \textbf{Step 2: } \textit{Compute pairwise \(W_2\) similarity of segments}
\State \(D \gets \left[\exp\left(-\frac{1}{2}W_2^2(S_i, S_j)\right)\text{ for } S_i, S_j \in \mathcal{S}\right ]\) \Comment{Compute similarity matrix}
\State \textbf{Step 3: } \textit{Cluster segments based on similarity matrix}
\State \(L \gets \text{Cluster}(D)\)
\State \textbf{Step 4: } \textit{Label points according to the label of their parent segment}
\State \(l_i \gets [L(S(X_{t,i})) \text{ for } t=1:T]\)
\EndFor
\State \(\text{PointLabels}\gets [l_1(t)l_2(t)\dots l_d(t) \text{ for } t=1:T]\) \Comment{Concatenate labels to make final assignment}
\State \Return PointLabels
\end{algorithmic}
\end{algorithm}

We summarize the above with the following proposition:
\begin{proposition}
For fixed \(w\) and \(q\), for inputs \(\{\vec{X_{t}}\}_{t=1}^T \subset \R^D\), the runtime complexity of Algorithm 1 is \(\mathcal{O}(TD).\)
{Fix \(w \in \{1, \dots \lfloor T/2\rfloor\}, q \in (0,1)\)}, and assume that the number of segments \(N\) scales proportionally to the length of the input data \(X = \{X_t\}_{t=1}^T\).
\end{proposition}
\begin{proof}
    The claim follows directly from the construction of the algorithm: for each \(t = w, \dots, T - w\), we perform two sorting operations, which are \(\mathcal{O}(w\log w)\), but \(\mathcal{O}(1)\) with respect to the size and dimension of the input data. Likewise, the subroutine identifying the inflection points of the identified segments is no more complex than sorting the empirical distribution of each segment. Since the number of segments is assumed to grow linearly in the number of data points, we amortize by considering the expected length of each segment, which will be \(T/N\). Performing \(N \propto T\) sorting jobs on a list of size \(T/N\) has complexity \(\mathcal{O}(N\frac{T}{N}\log(\frac{T}{N}) = \mathcal{O}(T)\). Because this procedure is carried out over each of the \(D\) dimensions of the data, the algorithm has total asymptotic time complexity \(\mathcal{O}(TD).\)
\end{proof}

\section{Methodology \& Numerical Results}
\label{Method and results}
For each of the particular time series we examine going forward, we will apply our method and examine the point labels so produced. To keep things visualizable, we will look at timeseries which takes values in \(\R\), and the \(3\)-torus. While almost all the data sets are computer generated, some are more artificial than others --- the ``toy'' data set, for example, was constructed explicitly to serve as a testing ground for our procedure \footnote{Our code and data are available at \url{https://github.com/dcgentile/mfpy}}. The molecular dynamics simulations were retrieved from the deeptime Python library \cite{hoffmannDeeptimePythonLibrary2021} and the GROMACS library for chemical simulations \cite{abrahamGROMACSHighPerformance2015}. Our collection of data sets includes one physically generated trajectory, courtesy of the Naval Research Laboratory \cite{turgutMeasuredDepthdependenceWaveguide2016}.

\subsection{Generated Data Sets}
\subsubsection{Toy Data Set}
\label{ToyDataSets}
Our first data set comes from a family of artificially constructed trajectories meant to give our procedure an idealized setting in which to work. We fix two means \(m_1, m_2 \in \R\) and a variance \(\sigma > 0\) to determine two separated Laplace distributions \(\mu_1, \mu_{2}\), and we fix a transition length \(L\), a total number of samples \(T\) and collection of change points \(\left\{ t_i \right\} \subset \left\{ 1\dots T \right\}\). The trajectory is then generated by sampling \(X_{t} \sim \mu_1\) for \(t = 0, \dots, t_1\), the Wasserstein interpolant of \(X_t \sim [\mu_1, \mu_2]_{t}\) for \(t = t_1 + 1, \dots, t_2\), and then \(X_t \sim \mu_2\) for \(t = t_2 + 1, \dots, t_3\), repeating the pattern until time \(T\). 

\begin{figure}[htbp]
    \centering
    
\begin{tabular}{ |p{3cm}|p{3cm}|p{3cm}|p{3cm}|  }
 \hline
 \multicolumn{4}{|c|}{Toy Laplacian Precision \& Recall} \\
 \hline
 \# True CPS & \# Detected CPS & Precision & Recall \\
 \hline
 38 & 38 & 0.89 & 0.89 \\
 \hline
\end{tabular}

\caption{Table of precision and recall scores for the toy data set. Precision and recall are averaged over tolerances of \(tol=0,\dots, 100\).}

\label{ToyFTable}
\end{figure}

We find that our procedure is able to identify the change points of the trajectory and accurately classify the resulting segments, and we summarize the output in relation to the ground truth data in table \ref{ToyFTable}. Because of the high variance in the timeseries, we do not expect our algorithm to output a set of change points which is identical to the ground truth, but rather a set where each candidate change point is close to a ground truth element. For this reason, we average the precision and recall scores over a range of tolerances.

%\begin{figure}[htbp]\label{toy-adpc}
%    \centering 
%    \begin{subfigure}[t]{0.49\textwidth}
%        \includegraphics[width=1\linewidth]{./figs/toy/toy-cps.pdf}
%        \caption{Detecting change points in an idealized setting. Top: we generate a synthetic data set by sampling from a fixed Laplace distribution with mean \(100\) and variance \(20\), and one with mean \(200\) and variance \(20\). We simulate transitions between the two states by sampling from the Wasserstein geodesic (McCann interpolant) between the two states. Bottom: change points are superimposed on the trajectory as red dashed lines. }
%    \end{subfigure}
%    \begin{subfigure}[t]{0.49\textwidth}
%        \centering
%        \includegraphics[width=1\linewidth]{./figs/toy/toy-adpc-clustering.pdf}
%        \caption{Segment cluster labels are learned in an unsupervised manner using the clustering by advanced density peaks algorithm. In this case, the clustering scheme identifies both metastable states and outputs a separate label for the transition regions.}
%    \end{subfigure}
%    \caption{Toy Laplace-based Trajectory Segmentation and Clustering}
%\end{figure}

\begin{figure}[htbp]\label{toy-adpc}
    \centering 
        \includegraphics[width=1\linewidth]{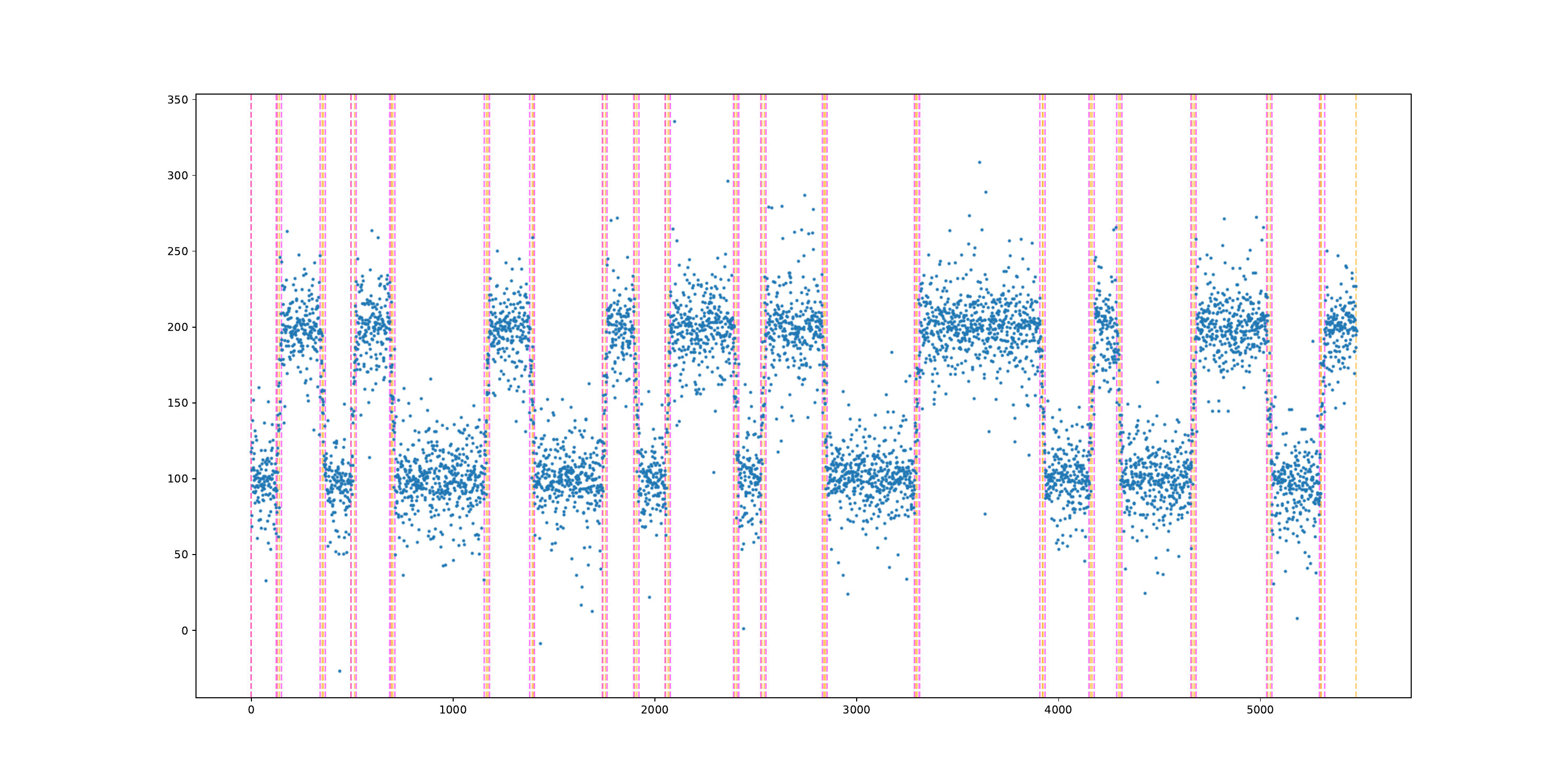}
        \caption{Detecting change points in an idealized setting. We generate a synthetic data set by sampling from a fixed Laplace distribution with mean \(100\) and variance \(20\), and one with mean \(200\) and variance \(20\) in an alternating pattern. We simulate transitions between the two states by sampling from the Wasserstein geodesic (McCann interpolant) between the two states; change points are superimposed on the trajectory as red dashed lines. }
\end{figure}
\begin{figure}[htbp]
        \centering
        \includegraphics[width=1\linewidth]{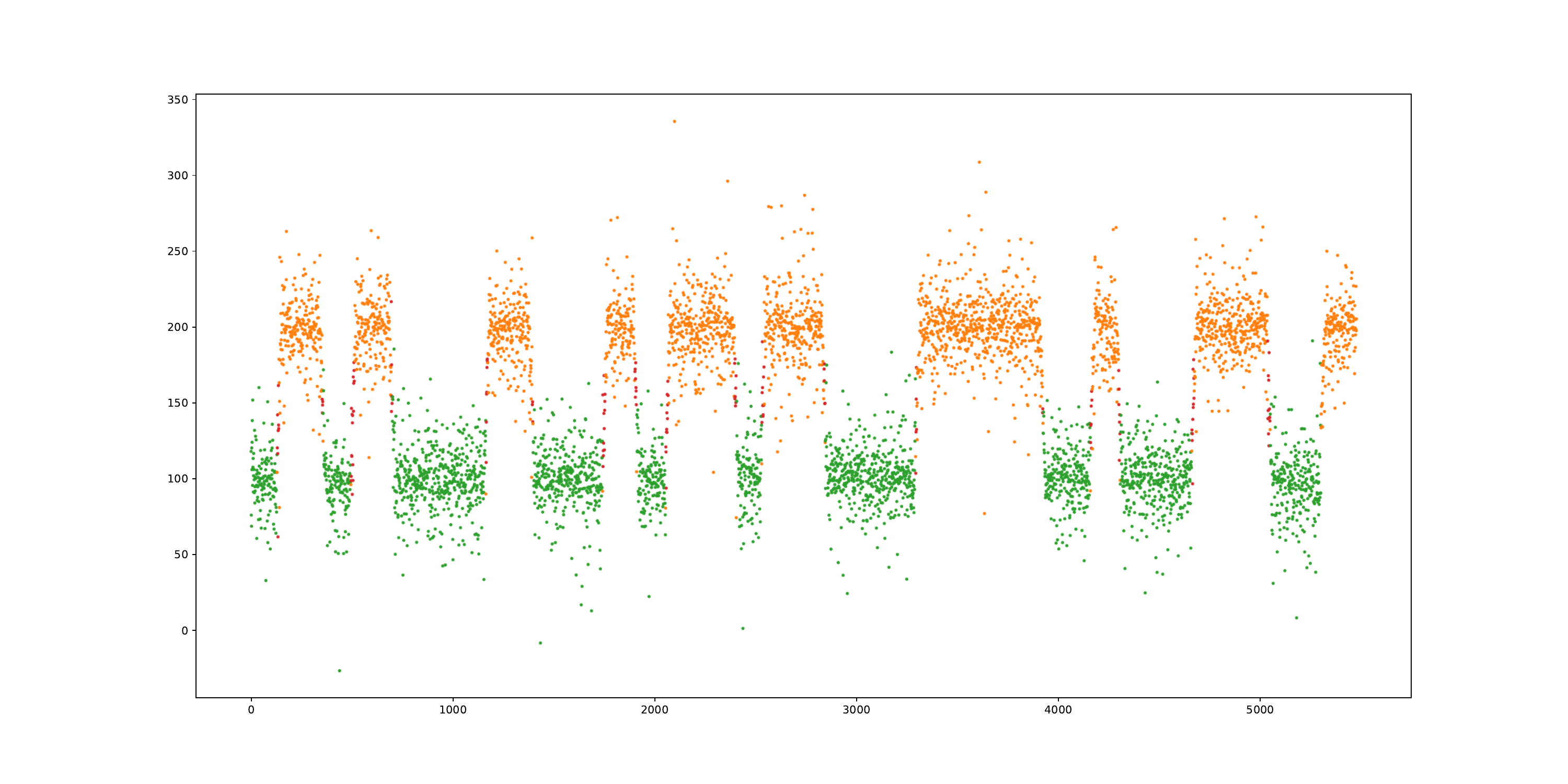}
        \caption{Segment cluster labels are learned in an unsupervised manner using the clustering by advanced density peaks algorithm. In this case, the clustering scheme identifies both metastable states and outputs a separate label for the transition regions. Parameters used to obtain this clustering were \(w=25, q=0.95.\)}
    \caption{Toy Laplace-based Trajectory Segmentation and Clustering}
\end{figure}

\begin{figure}[htbp]\label{toy-kmeans-adpc}
    \centering 
        \includegraphics[width=0.8\linewidth]{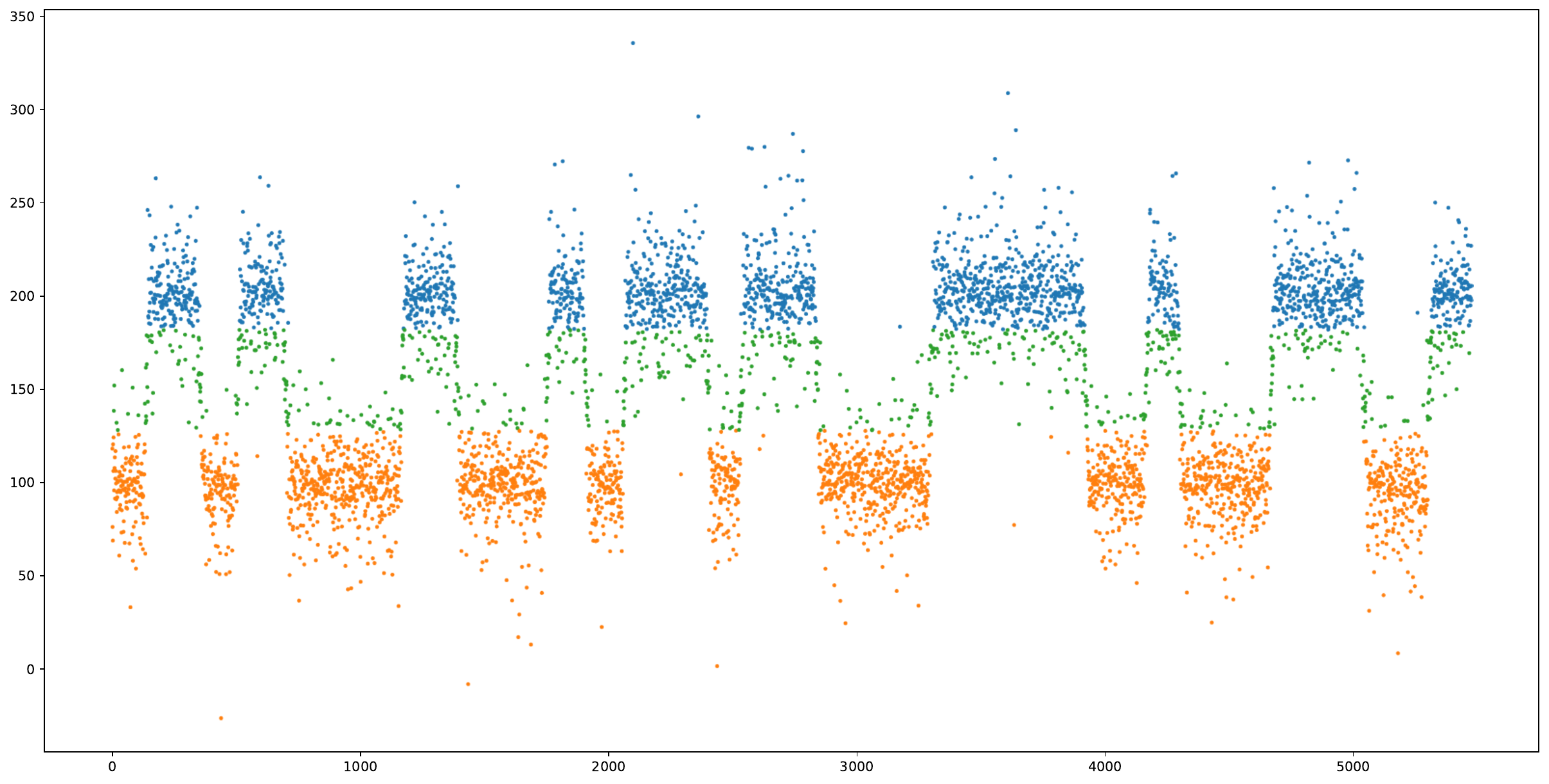}
        \caption{To compare our results against standard ``off-the-shelf" clustering methods, we cluster our data set with K-means. The resultant clustering fails to reflect the variance of the true underlying distributions and exhibits hard dividing hyperplanes.}
\end{figure}
\begin{figure}\label{toy-distributions}
\centering
\includegraphics[width=0.75\linewidth]{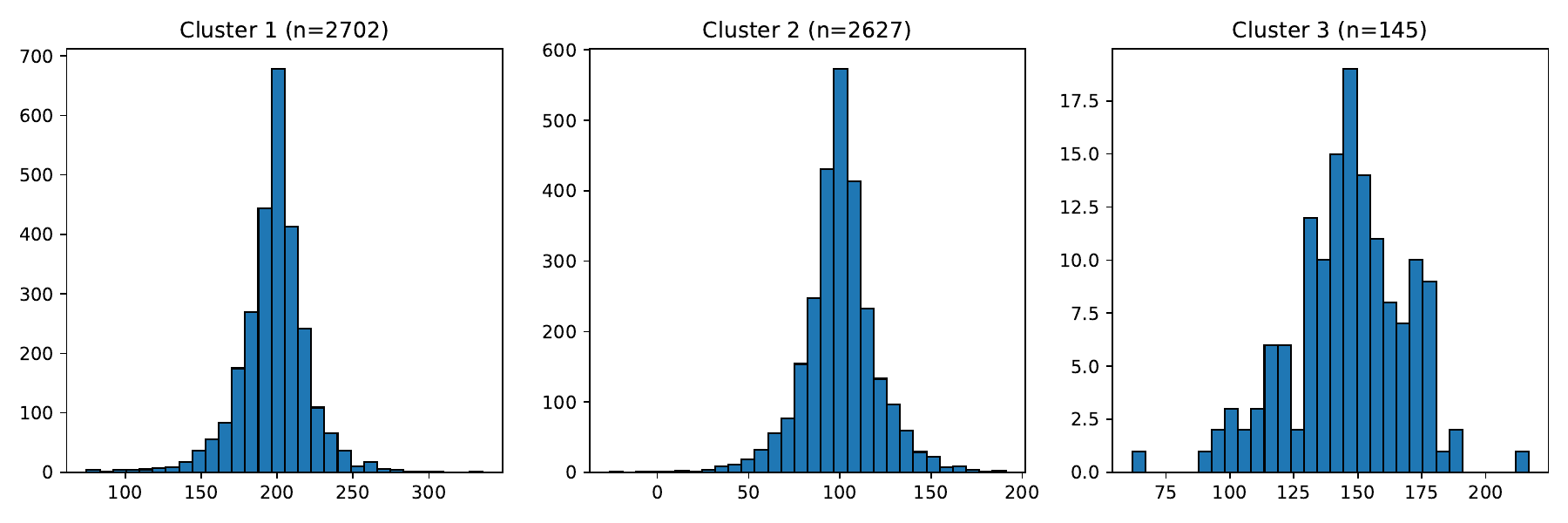}
\caption{Empirical distributions of the clusters in the toy data set learned by segment-based clustering with ADPC.}
\end{figure}
\begin{figure}\label{toy-kmeans-distributions}
\centering
\includegraphics[width=0.75\linewidth]{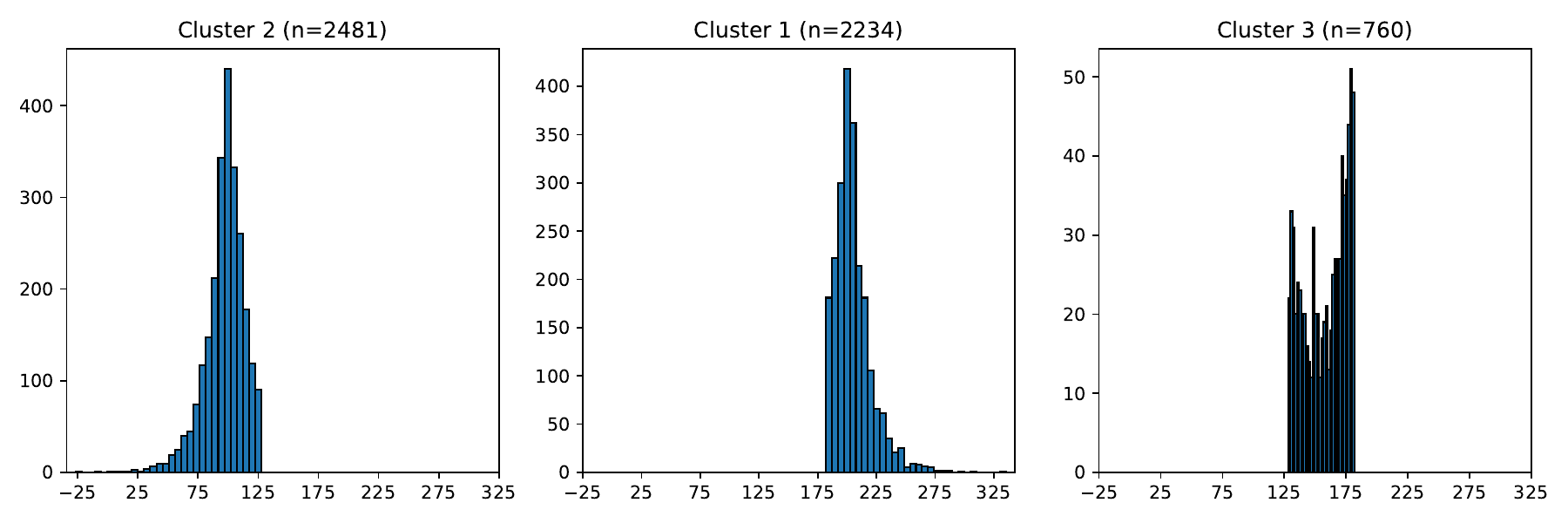}
\caption{Empirical distributions of the clusters in the toy data set learned by K-means.}
\end{figure}
\label{Langevin}
\subsubsection{Langevin Trajectory}
Next, we investigate the application of our algorithm to a sample Langevin trajectory. The energy landscape governing the particle trajectory has two distinct wells, centered at \(+1\) and \(-1\) with a barrier between the two which we can think of as regulating how difficult it is for the particle to transition from one state to the other.  The trajectories were simulated using OpenMM \cite{eastmanOpenMM7rapid2017} with particle mass set to 1amu, temperature set to 500K, and damping coefficient of \(10\text{ps}^{-1}.\)

We observe that after segmenting the trajectory with our algorithm and clustering via ADPC, three distinct states are identified, with one corresponding to each well, and the third capturing the transition between the two metastable states.

%\subsection{Langevin Trajectory}
\begin{figure}[htbp]\label{lang-cps}
\centering
%\begin{subfigure}[t]{0.45\textwidth}
\includegraphics[width=1\linewidth]{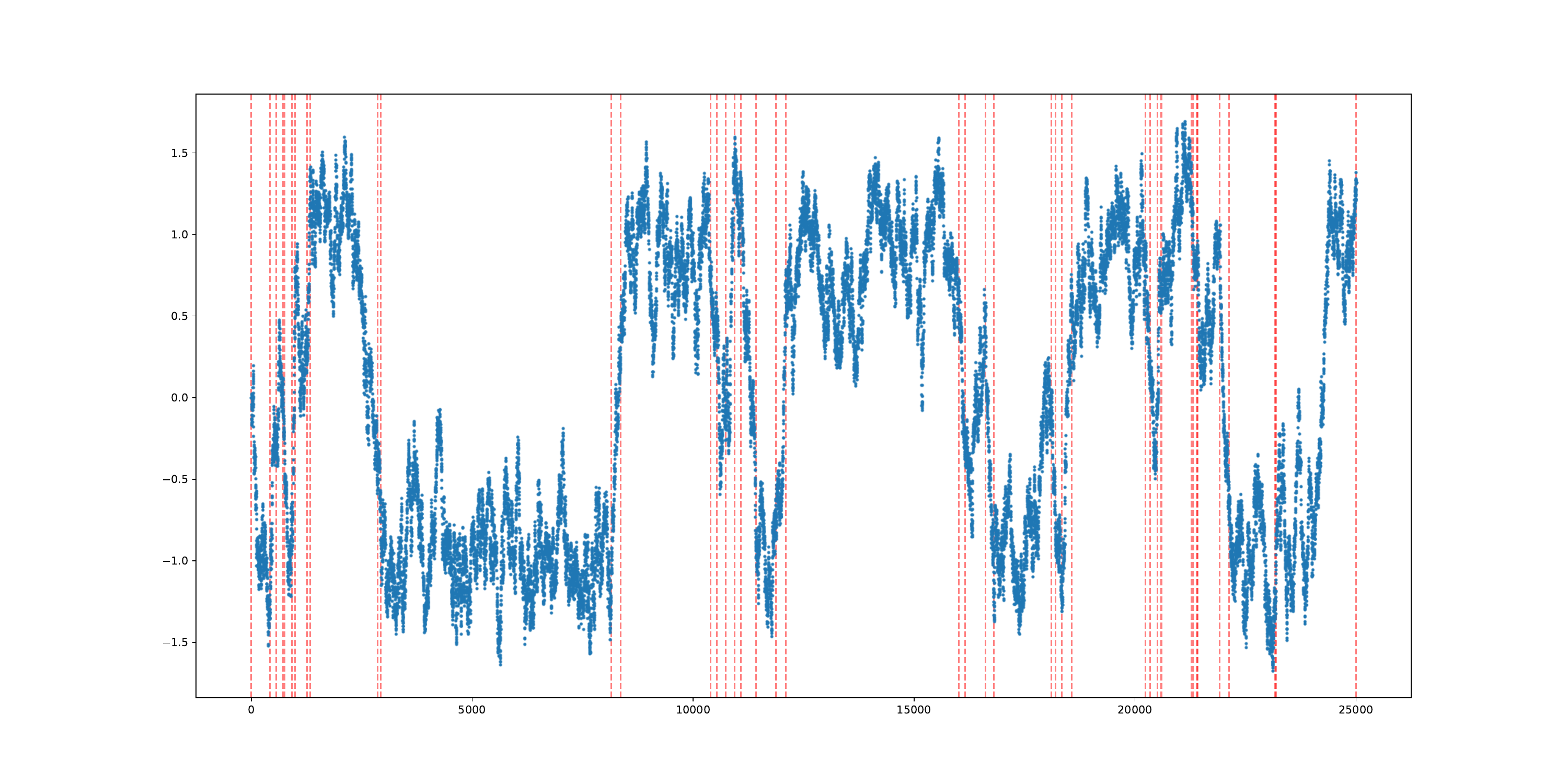}
\caption{The first 25,000 steps in a Langevin trajectory annotated with detected change points.}
%\end{subfigure}
\end{figure}
\begin{figure}[htbp]\label{lang-clusters}
\centering
%\begin{subfigure}[t]{0.45\textwidth}
\includegraphics[width=1\linewidth]{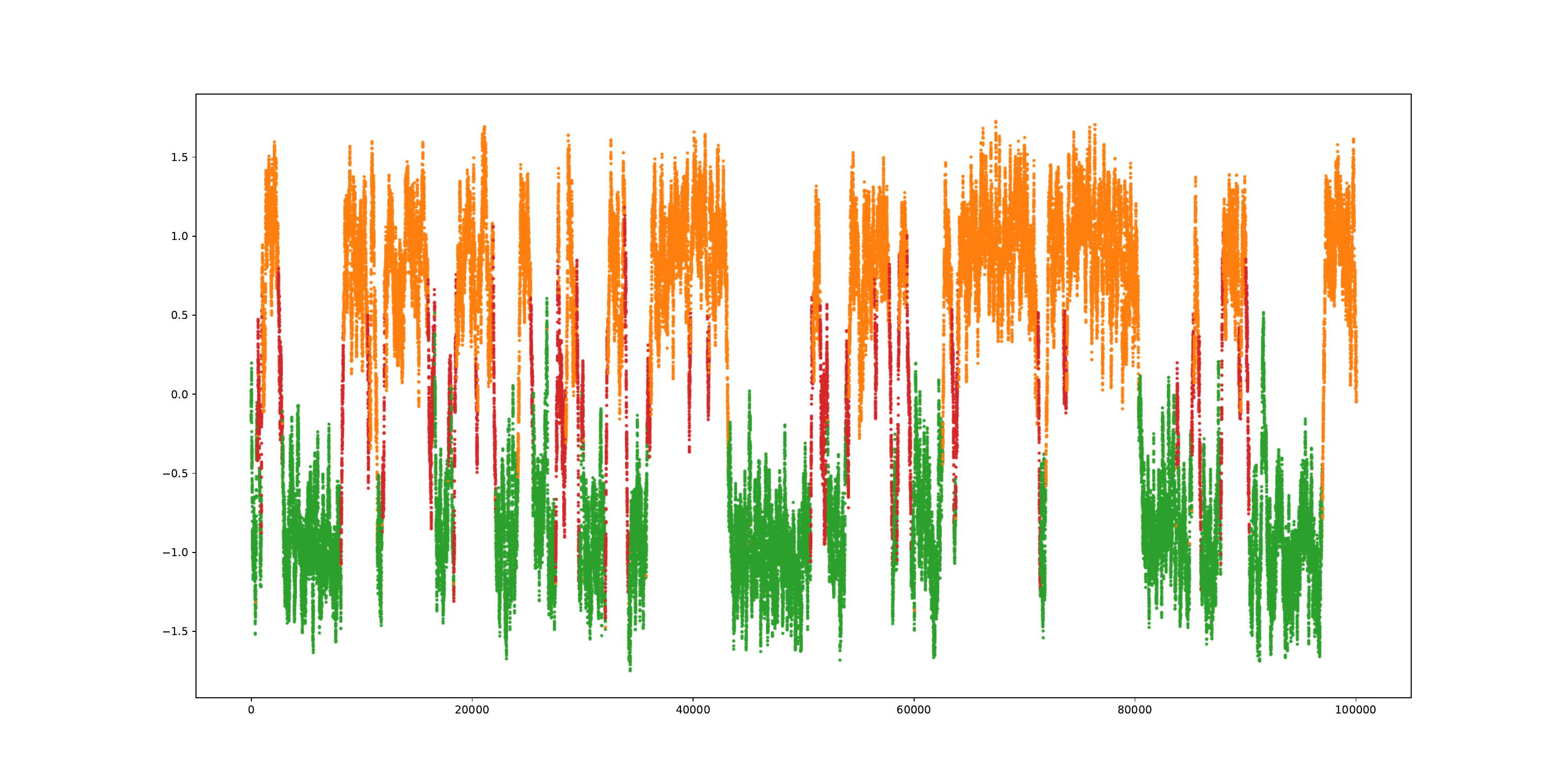}
\caption{Segment based clustering of a Langevin trajectory using advanced density peaks clustering. Parameters used to obtain this clustering are \(w = 215, q=0.86.\)}
%\end{subfigure}
    \caption{Langevin Trajectory Segmentation and Clustering}
\end{figure}

%\begin{figure}\label{lang-spectral-clustered}
%\includegraphics[width=0.75\linewidth]{./figs/langevin/langevin-spectral-clustering.pdf}
%\caption{Segment based clustering of a Langevin trajectory using spectral clustering with $K=3$}
%\end{figure}

%\begin{figure}\label{lang-sc-distributions}
%\includegraphics[width=0.5\linewidth]{./figs/langevin/langevin-spectral-distributions.pdf}
%\caption{Empirical distribution of clusters identified by segment-based clustering with spectral clustering}
%\end{figure}

\begin{figure}\label{lang-adpc-distributions}
\centering
\includegraphics[width=0.75\linewidth]{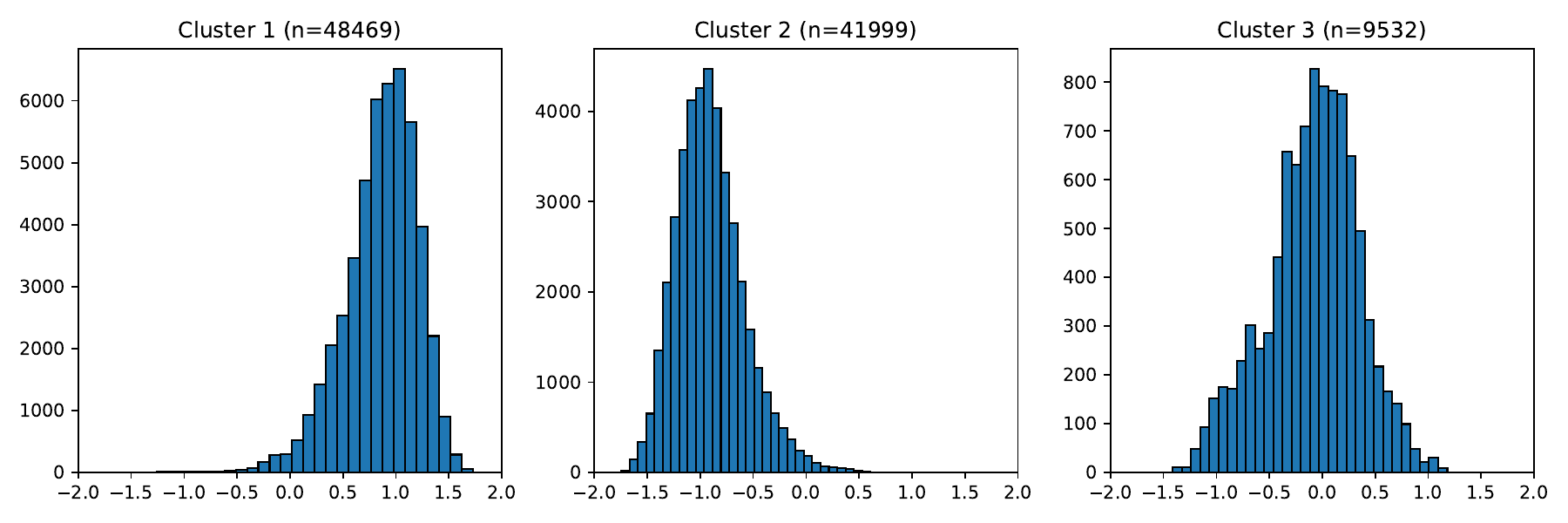}
\caption{Empirical distributions of clusters in the Langevin trajectory data set identified by segment-based clustering with advanced density peaks clustering}
\end{figure}

%\subsection{Prinz Potential Trajectory}
%\label{Prinz}

%%%
%\begin{figure}\label{prinz-spectral-clustered}
%\includegraphics[width=0.75\linewidth]{./figs/prinz/prinz-spectral-clustering.pdf}
%\caption{Trajectory of a particle in a Prinz potential clustered via segment-based clustering with a spectral clustering. The top four clusters by population are colored, while all points outside of these clusters appear in grey.}
%\end{figure}

%\begin{figure}\label{prinz-spectral-distributions}
%\includegraphics[width=0.5\linewidth]{./figs/prinz/prinz-spectral-distributions.pdf}
%\caption{Empirical distribution of the top four clusters identified via segment-based clustering using spectral clustering.}
%\end{figure}
%%%
\begin{figure}[htbp]\label{langevin-kmeans-adpc}
    \centering 
        \includegraphics[width=0.8\linewidth]{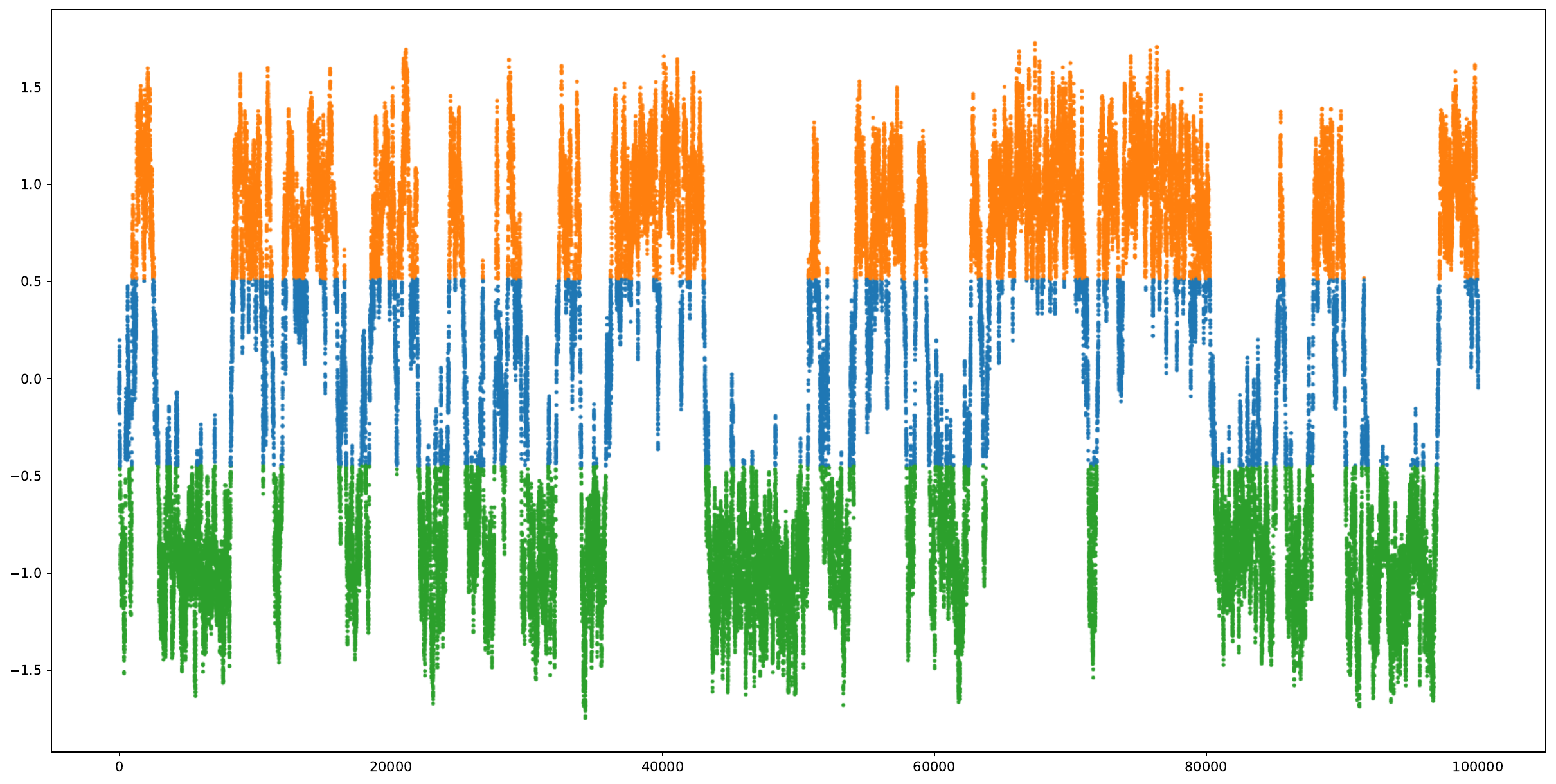}
        \caption{To compare our results against standard ``off-the-shelf'' clustering methods, we cluster our data set with K-means. The resultant clustering fails to reflect the variance of the true underlying distributions and exhibits hard dividing hyperplanes.}
\end{figure}
\begin{figure}\label{langevin-distributions}
\centering
\includegraphics[width=0.75\linewidth]{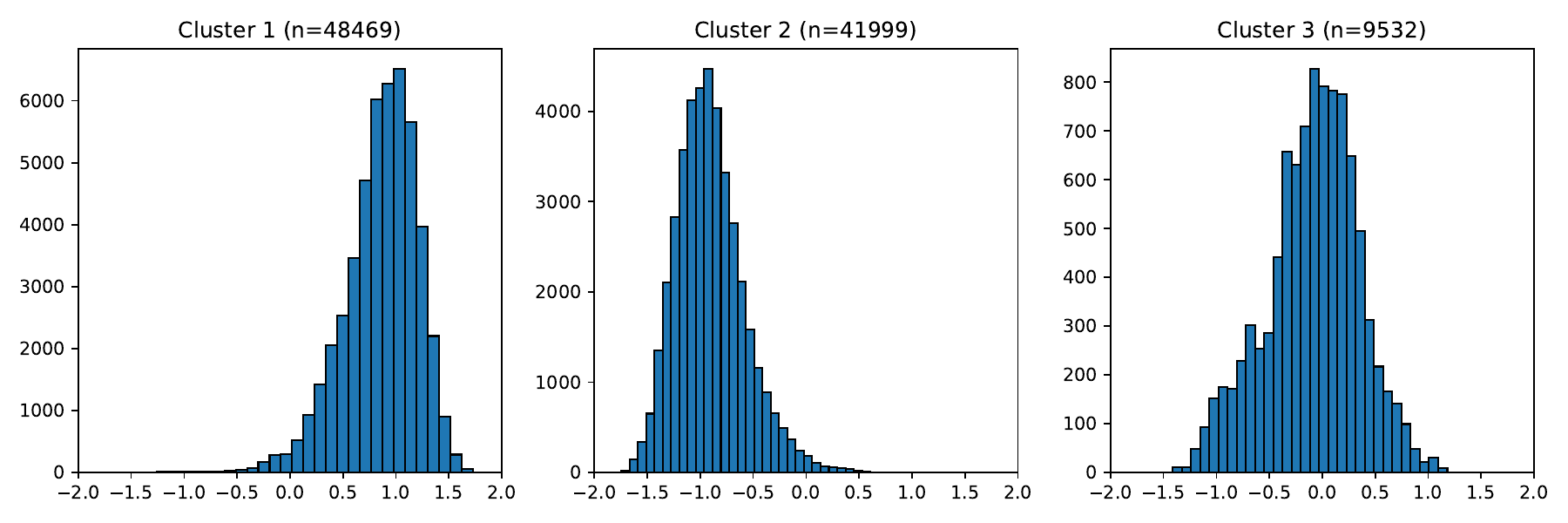}
\caption{Empirical distributions of the clusters in the Langevin data set learned by segment-based clustering with ADPC.}
\end{figure}
\begin{figure}\label{langevin-kmeans-distributions}
\centering
\includegraphics[width=0.75\linewidth]{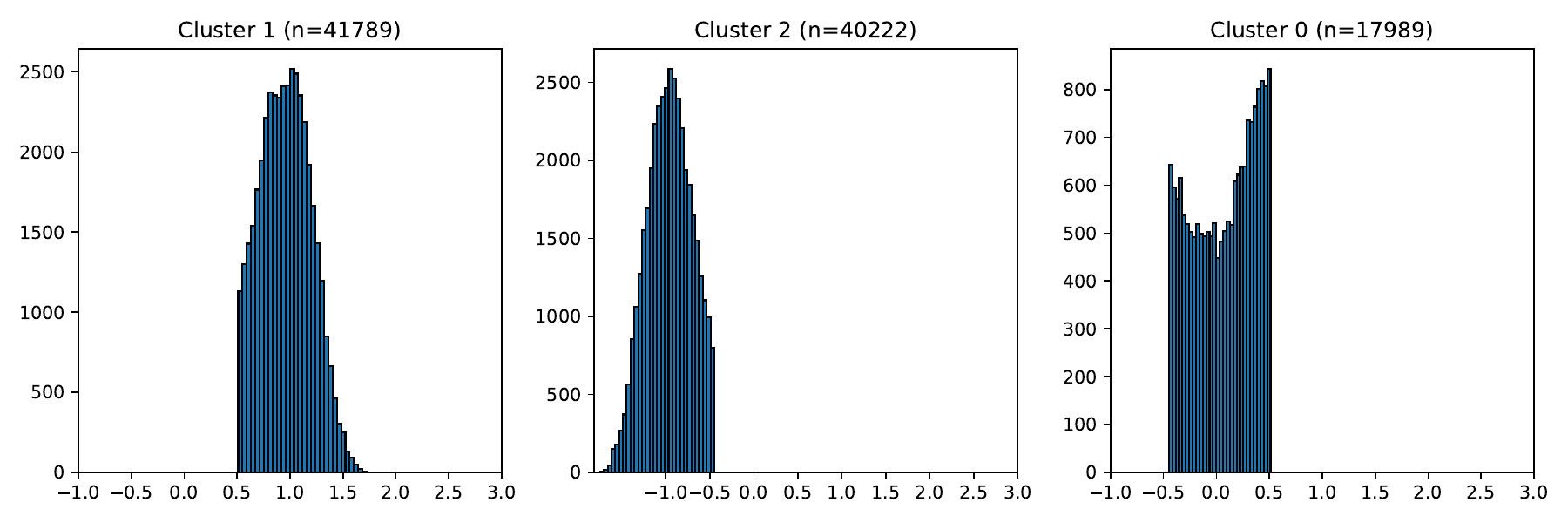}
\caption{Empirical distributions of the clusters in the Langevin data set learned by K-means.}
\end{figure}

\subsubsection{Particle in a Prinz Potential}
\begin{figure}\label{prinz-adpc-cps}
\centering
%\begin{subfigure}[t]{0.45\textwidth}
\includegraphics[width=1\linewidth]{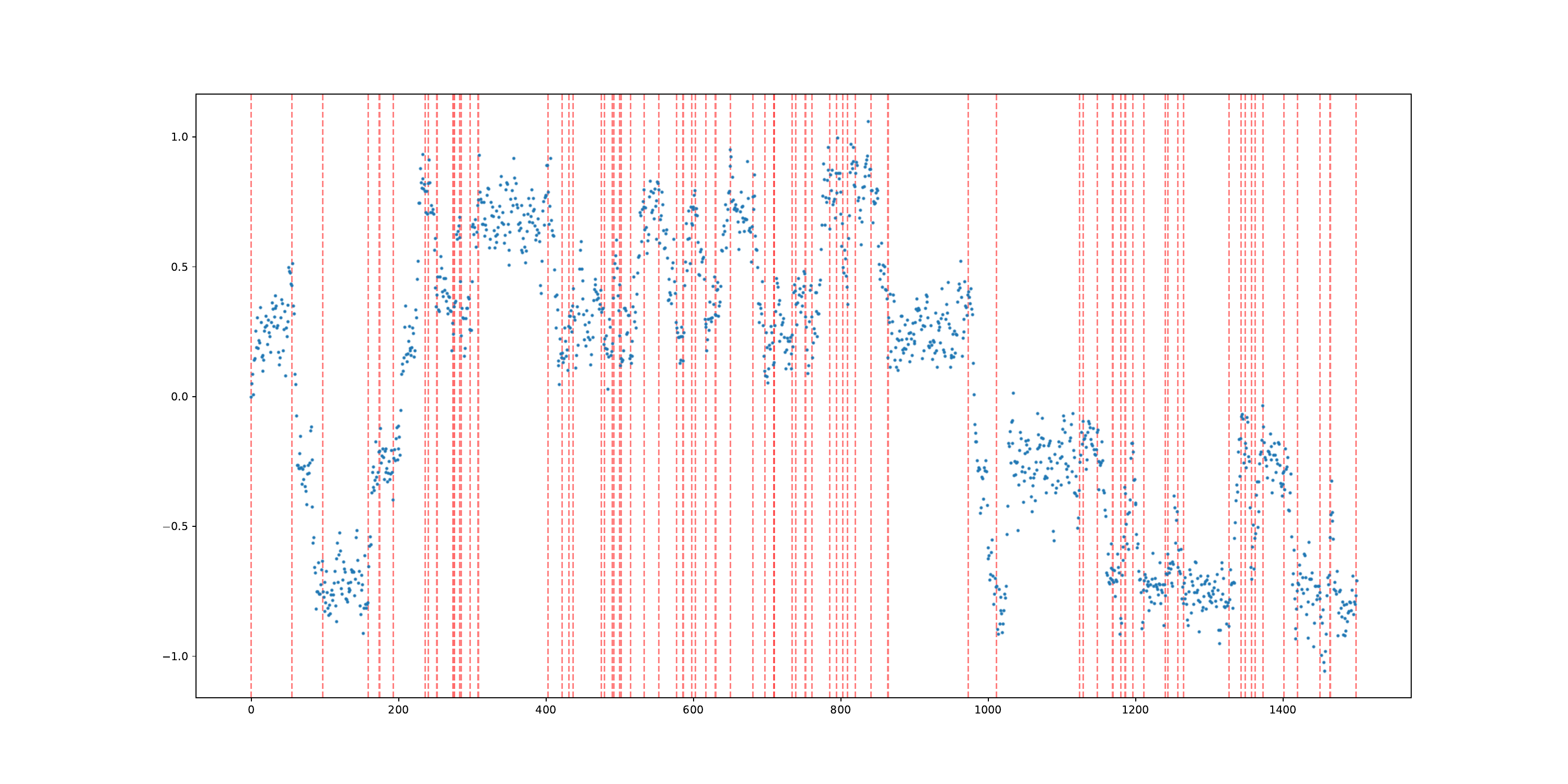}
\caption{Trajectory of a particle in a Prinz potential annotated with detected change points.}
    
%\end{subfigure}
\end{figure}
\begin{figure}\label{prinz-adpc-clustered}
\centering
%\begin{subfigure}[t]{0.45\textwidth}
\includegraphics[width=1\linewidth]{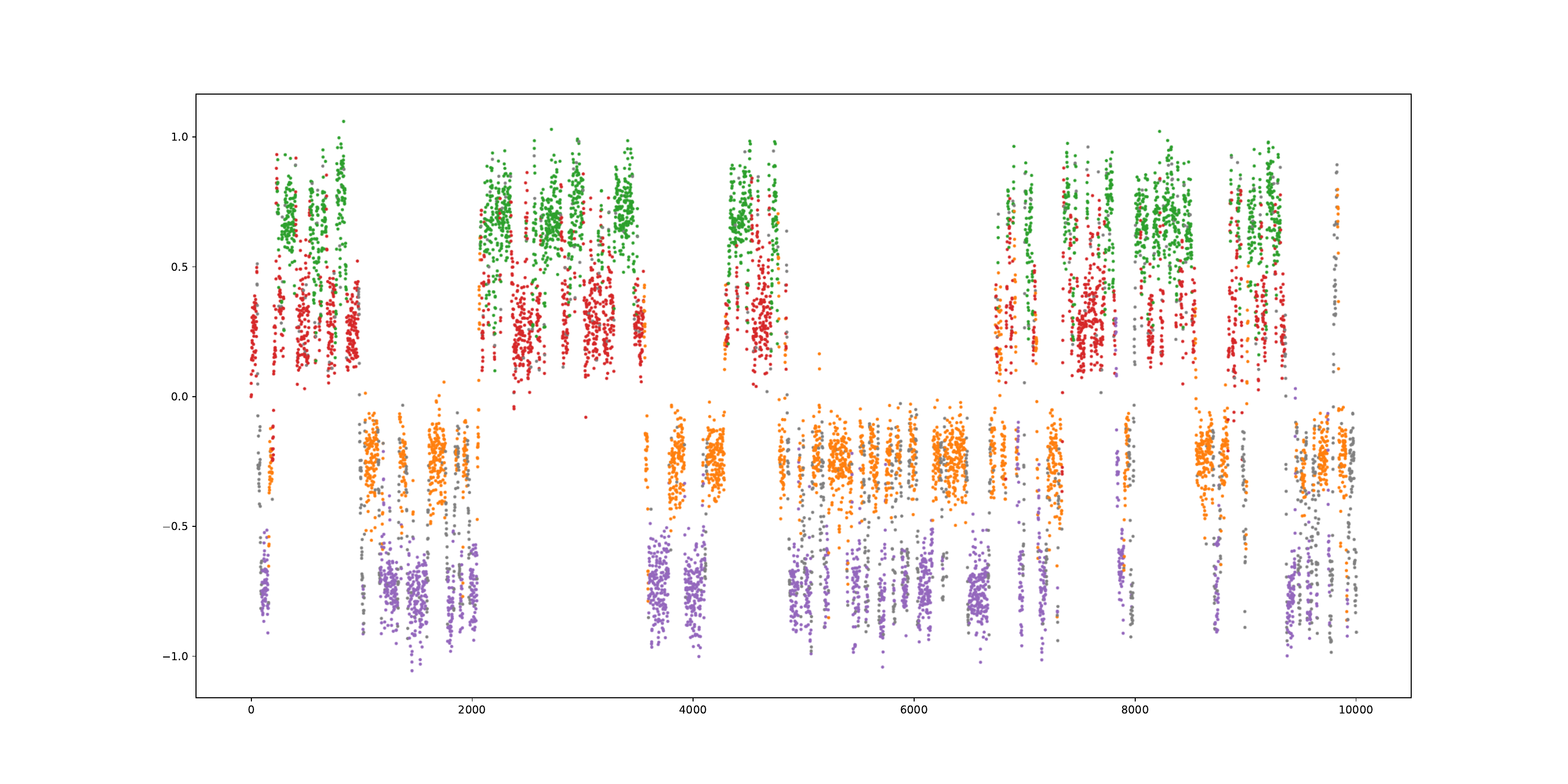}
\caption{Trajectory of a particle in a Prinz potential clustered via segment-based clustering with a advanced density peaks based clustering. The top four clusters by population are colored, while all points outside of these clusters appear in grey. Parameters used to obtain this clustering were \(w=16, q=0.5.\)}
%\end{subfigure}
\end{figure}
\begin{figure}\label{prinz-traj}
\includegraphics[width=\linewidth]{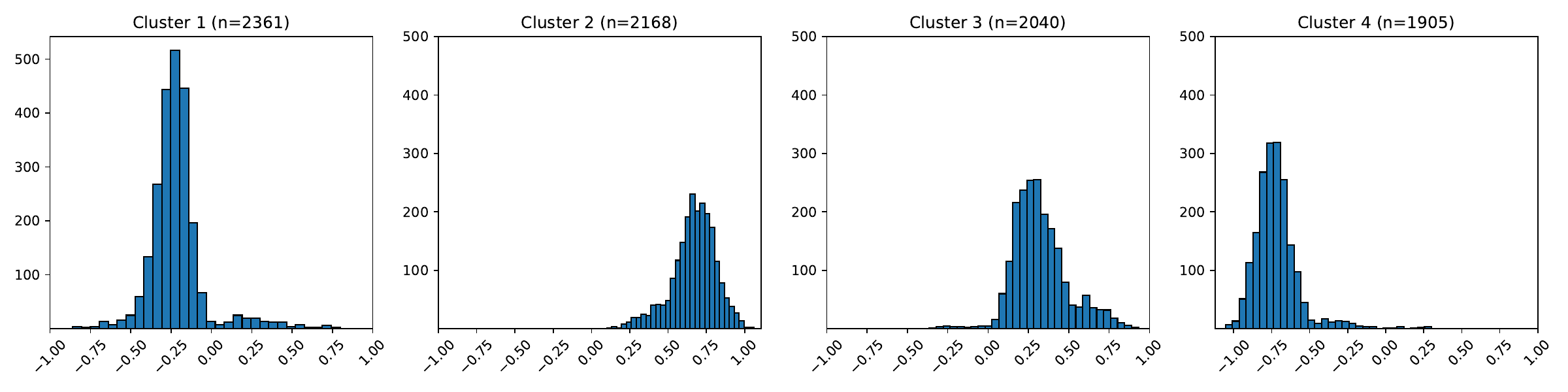}
\caption{Empirical distribution of the top four clusters in the Prinz potential data set identified via segment-based clustering using advanced density peaks based clustering.}
\end{figure}    
\begin{figure}\label{prinz-kmeans-distributions}
\includegraphics[width=1\linewidth]{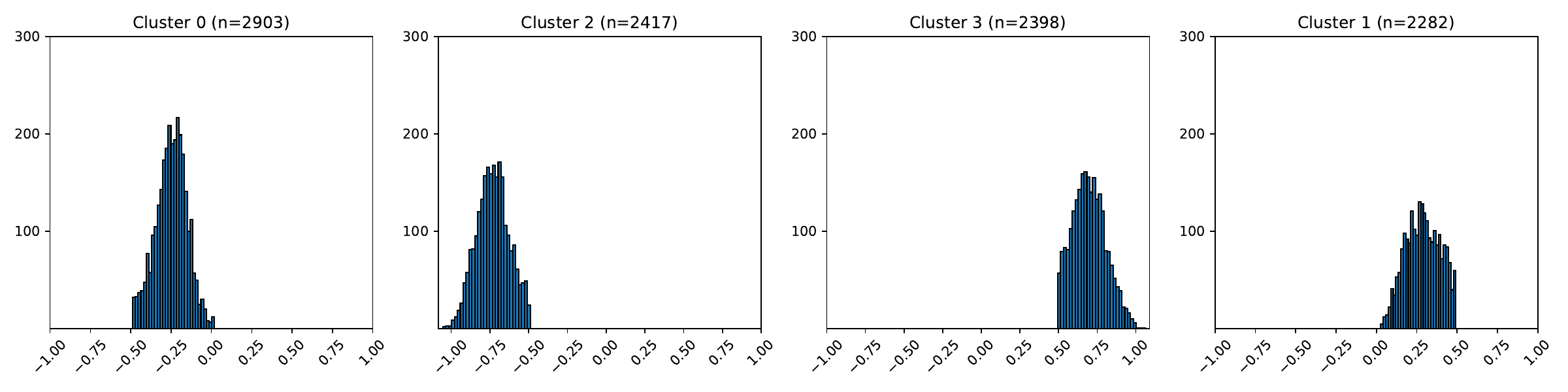}
\caption{Empirical distributions of the clusters in the Prinz data set learned by K-means.}
\end{figure}
Increasing the complexity of our example now, we apply our method to a particle trajectory in a Prinz potential. The trajectory is qualitatively similar to the previous Prinz trajectory --- there are four wells in the energy landscape that the particle can move between, with the height of the wall between each neighboring pair of wells differing somewhat. 

The trajectory is generated by an integrator of Euler-Maruyama type, with potential 
\[V(x) = 4\left(x^8 + 0.8e^{-80x^2} + 0.2e^{-80(x-0.5)^2} + 0.5e^{-40(x+0.5)^2}\right),\] and driven by the equation

\[x_{t+1} = x_t - \frac{h\nabla V(x_t)}{md} + \sqrt{\frac{2hkT}{md}}\eta_t\]

where \(\eta_t \sim \mathcal{N}(0,1)\) \cite{hoffmannDeeptimePythonLibrary2021}.

After segmenting the trajectory with our algorithm and clustering the the segments via ADPC, we find that five states are identified, one corresponding to each distinct well in the landscape, and an additional one that seems to correspond to the transition region between the two wells with the lowest barrier to transition.
\subsection{Molecular Dynamics Trajectories}
%\begin{figure}[htbp]
%\centering
%\begin{subfigure}[c]{0.45\textwidth}
%\includegraphics[width=\linewidth]{./figs/adp/adp-clusters.pdf}
%        \caption{Alanine dipeptide trajectory clustered with segment-based clustering with advanced density peaks.}
%\end{subfigure}
%\begin{subfigure}[c]{0.45\textwidth}
%\includegraphics[width=\linewidth]{./figs/adp/adp-top-individual-clusters.pdf}
%        \caption{Alanine dipeptide trajectory clustered with segment-based clustering with advanced density peaks.}
%\end{subfigure}
%    \caption{Alaline Dipeptide Clustering}
%\end{figure}
We now turn our attention to simulated molecular trajectories for Ace-Val-NMe (valine dipeptide, or VDP) which were generated with the GROMACS molecular simulations package \cite{abrahamGROMACSHighPerformance2015}. Typically, VDP trajectores are modeled via Markov State Modeling techniques \cite{bowmanIntroductionMarkovState2014}, where the data is clustered using a technique like PCCA++ \cite{roblitzFuzzySpectralClustering2013}, and then a reversible transition matrix is estimated based on the cluster identifications. Such trajectories are well-studied in the computational chemistry community and are commonly used for benchmarking algorithms in that domain \cite{damjanovicCATBOSSClusterAnalysis2021}.
Modern machine learning techniques are also employed to model these trajectories, wherein an optimal basis of nonlinear functions to describe the model is found via deep learning \cite{mardtVAMPnetsDeepLearning2018}, \cite{wuVariationalApproachLearning2020}.
Although these trajectories are simulated with several dozen spatial dimensions (VDP contains 28 atoms and hence in principle requires 84 dimensions to express the coordinates of each atomic position), corresponding to the atoms which constitute the molecule, it is well-known that the dimension of a VDP trajectory can be reduced down to \(d=3\), by tracking the so-called ``dihedral'' angles, known as \(\phi\), \(\psi\), and \(\chi\), which are periodic and hence the trajectory can be embedded onto the \(3\)-torus. 

\begin{figure}
    \centering
    \includegraphics[width=0.5\linewidth]{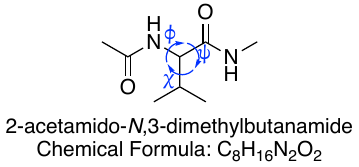}
    \caption{Structure of Ace-Val-NMe, annotated with backbone dihedral angles.}
\end{figure}

Consequently, this means that our approach must now be applied to a setting which is non-Euclidean.
This does not in principle impose a change to our approach, but it does mean that we can no longer rely on (\ref{ot1d}) for a computational boost, and instead we must deal with transport costs on the circle \cite{hundrieserStatisticsCircularOptimal2022}.
Our segmentation method allows ADPC to identify 9 distinct clusters which are clearly separated. 

Another data set of interest, and one which is, like VDP, commonly used for benchmarking MDS algorithms, is Ace-Ala-NMe (alanine dipeptide, or ADP). Unlike our previous data sets, ADP provides difficulties for our approach because of the rapid way it transitions between its metastable states, resulting in an undersampling problem.
By this we mean that among the collection of metastable states of interest, some appear infrequently and for short bursts of time, and even when trying to identify change points by hand for small portions of the trajectory, we find that segments contain sometimes as few as a dozen points. 
As discussed earlier, this would imply that it is necessary to take a very small window size \(w\) to avoid windows which contain multiple distinct states for a given angle, but of course choosing a \(w\) small means that the segments we see tend to be poor empirical representations of any cumulative distribution function.

Thankfully, VDP is a somewhat more cooperative MD trajectory, as we can observe in the relative stability of each of its angular trajectories. The 9 most signficantly populated clusters identified by our pipeline are clearly separated and account for approximately \(82\%\) of the total data, with the remainder belonging to still smaller clusters. Because VDP does not experience the slow and frequent transitions that ADP does, the cluster 
separation obtained for the VDP data set is much cleaner and is also readily apparent when examined at the componentwise level. Looking at the individual components of the VDP data, we see that the pipeline is identifying relatively subtle differences between states which are consistent in their appearance but with overlapping support. This is most pronounced in the \(\psi\) angle, where the algorithm identifies 3 distinctive clusters supported in the vicinity of \([-60,60]\) and 3 more distinctive clusters supported in the vicinity of \([60, 210]\).  

\begin{figure}[htbp]
    \centering
    \begin{subfigure}[b]{0.32\textwidth}
        \centering
        \includegraphics[width=\linewidth]{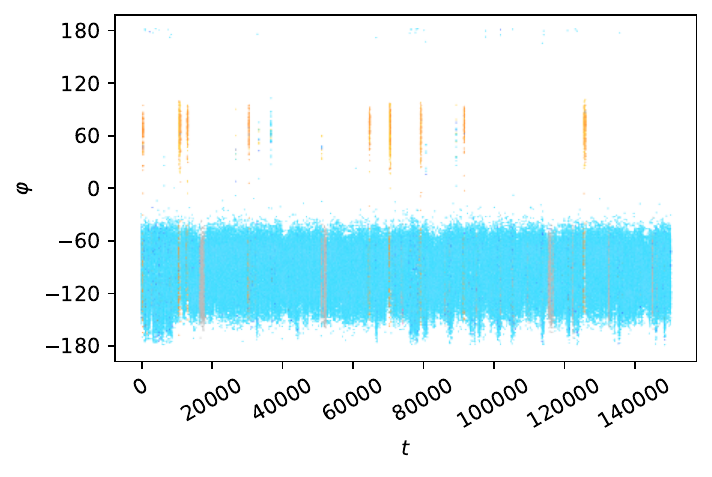}
        \caption{$\phi$ angle}
    \end{subfigure}
    \hfill
    \begin{subfigure}[b]{0.32\textwidth}
        \centering
        \includegraphics[width=\linewidth]{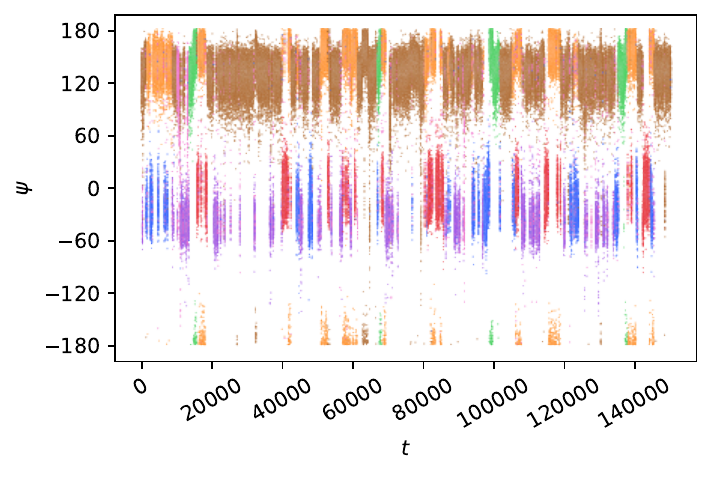}
        \caption{$\psi$ angle}
    \end{subfigure}
    \hfill
    \begin{subfigure}[b]{0.32\textwidth}
        \centering
        \includegraphics[width=\linewidth]{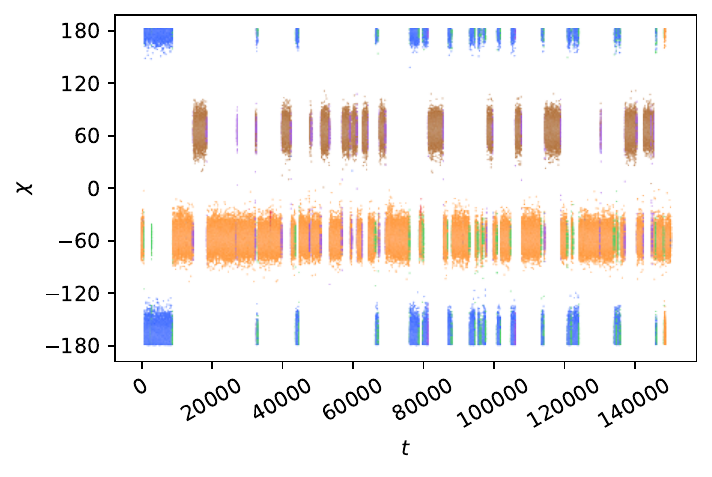}
        \caption{$\chi$ angle}
    \end{subfigure}
    \caption{VDP dihedral angles labeled via segment-based clustering with circular optimal transport. Each angle is put through the process of CPD and segment clustering. Then clusters are combined componentwise to produce a final labeling.}
    \label{fig:vdp-angles}
\end{figure}

\begin{figure}[htbp]
\centering
\begin{subfigure}[c]{0.45\textwidth}
        \includegraphics[width=0.9\linewidth]{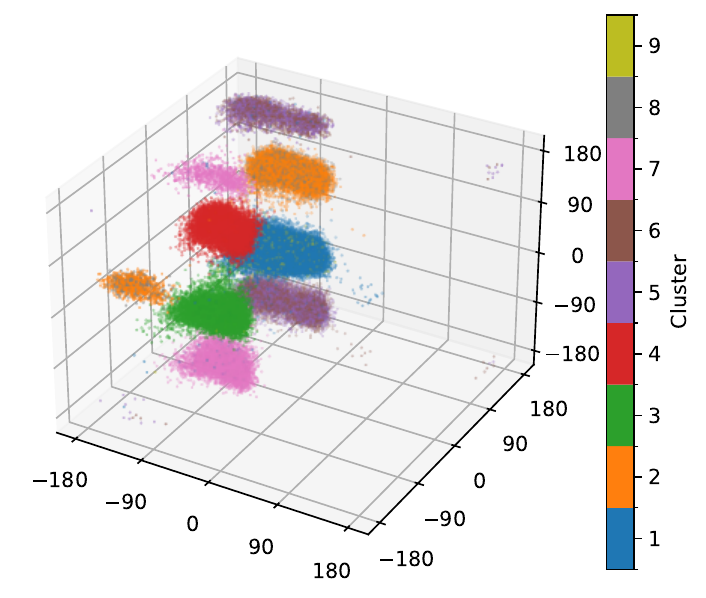}
        \caption{Final label assignments for valine dipeptide trajectory clustered with segment-based clustering with advanced density peaks. The parameters used for each angle are \(w = 50, q=0.95\)}
\end{subfigure}
\begin{subfigure}[c]{0.45\textwidth}
\includegraphics[width=1.1\linewidth]{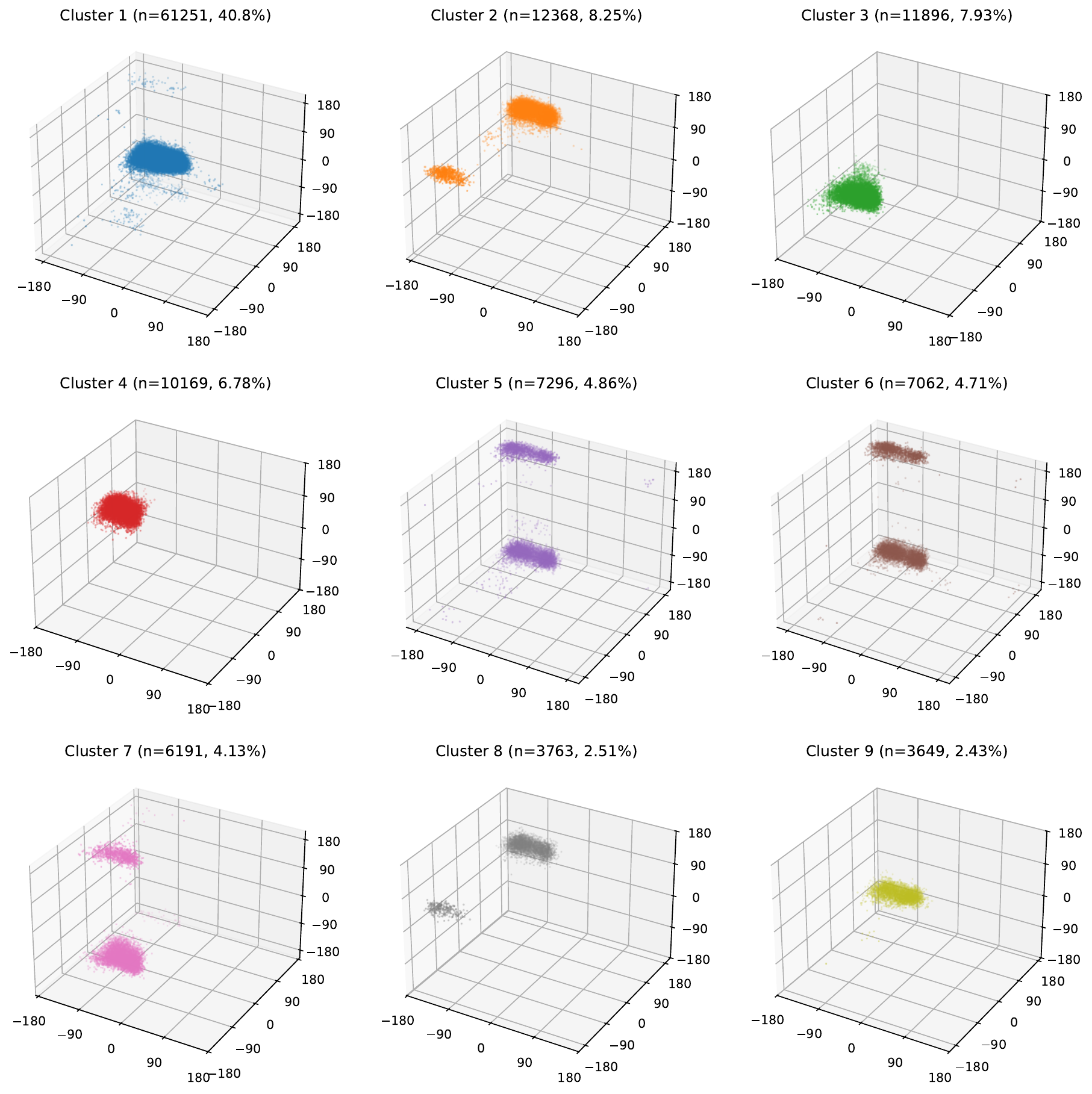}
        \caption{Valine dipeptide trajectory clustered with segment-based clustering with advanced density peaks. N.B.: percentages do not sum to 1, since only the top 9 most populous clusters are accounted for. Together these clusters account for \(82.43\%\) of the data set.}
\end{subfigure}
    \caption{Valine Dipeptide Clustering}
\end{figure}

\begin{figure}
    \centering
    \includegraphics[width=0.8\linewidth]{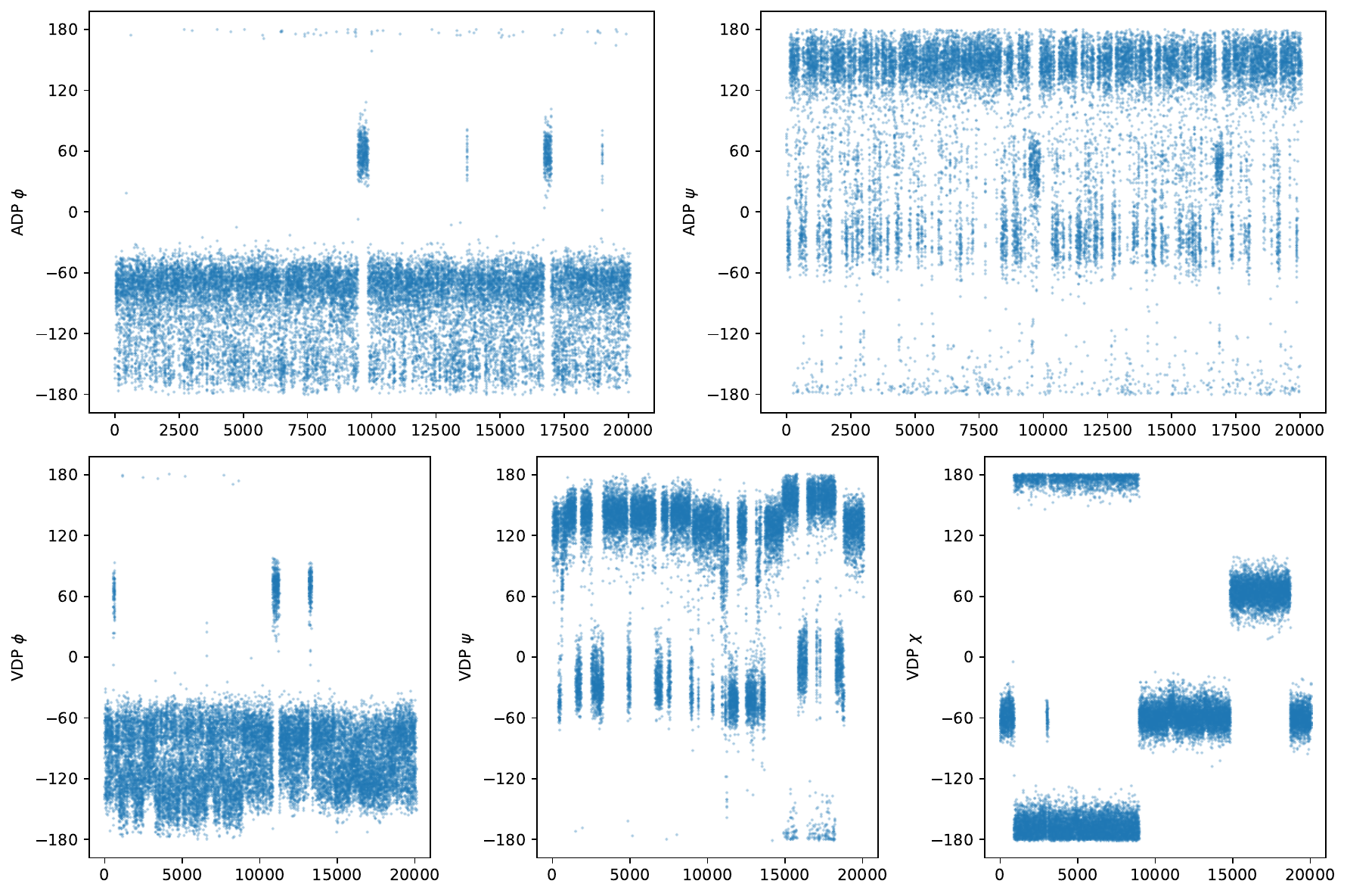}
    \caption{Comparison between the componentwise angles of ADP and VDP. Our algorithm struggles to handle ADP, as it bounces back and forth between states very rapidly, with a large amount of overlap between states, something which is particularly pronounced in the \(\psi\) angle of ADP (top right). VDP is a more tractable data set for us because of the cleaner and longer-lived changes.}
    \label{fig:placeholder}
\end{figure}

\subsection{Underwater Acoustics}
\label{UWA}

 Our real-world data set comes from an experiment conducted by the Naval Research Laboratory \cite{turgutMeasuredDepthdependenceWaveguide2016}. For this experiment, two underwater speakers were towed behind a vessel emitting four distinct signals, each repeated twice, over the course of 9 seconds.  Three of these samples were linear frequency modulations (LFM), wherein the signal starts at one frequency and increases linearly to another terminal frequency, and one was a continuous waveform. 
 
 We find that a modification of our algorithm applied to the spectrogram of the data is able to identify a suitable selection of change points for the samples. The simplification of our approach in this setting occurs because computing the spectrogram of a signal such as this naturally results in a sequentially ordered set of power spectral densities.  This differs in character from the reduction that occurs in mapping the high-dimensional atomic coordinate representation of VDP onto the \(3\)-torus. In that case, there is a one-to-one correspondence between samples represented in full atomic coordinates and samples represented by dihedral backbone angles. In the case of our underwater acoustics experiment, a different sort of compression occurs. Rather than computing the Wasserstein distances between the sampled voltages read by an analog-to-digital converter, we instead use a short-time Fourier transform (STFT) to examine the change in the frequency domain content overtime. The spectrogram obtained by computing the STFT of the raw voltage data, using a Tukey window with a shape parameter of 0.25 \cite{virtanenSciPy10Fundamental2020}, is an array where the columns represent the distribution of frequency content in the signal over some window of time, while the rows represent the concentration of frequency content for a certain band of frequencies over time. The resolution of the spectrogram is a function of the number of samples per window used in the STFT; in this case, we use \(n=512\). Overall this results in a significant reduction in the number of Wasserstein distances computed, since rather than sliding across every sample of the raw data, we now simple compute the Wasserstein distance between contiguous columns of the spectrogram. In \ref{spectrogram}, the raw number of samples in the time domain is \(421,848\), but the number of power densities to compare after computing the spectrogram is just \(941\). Because ground truth data is available, we also produce a table of precision and recall scores for our algorithms output as we did for the artificially generated toy-data set, with the slight modification that the tolerance now scales with number of samples per Fourier window, which is here taken to be \(n=512\). The sampling rate of the data is \(11718\) samples per second, thus the corresponding tolerance range can be interpreted as detecting a change within roughly \(0.1-1.0\) seconds.
 
\begin{figure}
    \centering
    
\begin{tabular}{ |p{2.5cm}|p{2.5cm}|p{2.5cm}|p{2.5cm}|p{2.5cm}|p{2.5cm}|}
 \hline
 \multicolumn{6}{|c|}{Underwater Acoustics Data Precision \& Recall} \\
 \hline
 \# True CPS & \# Detected CPS & Precision & Recall  & Max Precision & Max Recall\\ 
 \hline
 31 & 36 & 0.513  & 0.597 & 0.694 & 0.806 \\
 \hline
\end{tabular}

\caption{Table of precision and recall scores for the toy data set. Precision and recall are averaged over tolerances of \(tol=1536,\dots, 12800\)}

\end{figure}

\begin{figure}[htbp]
\begin{subfigure}[t]{0.45\textwidth}
    \includegraphics[width=1.1\linewidth]{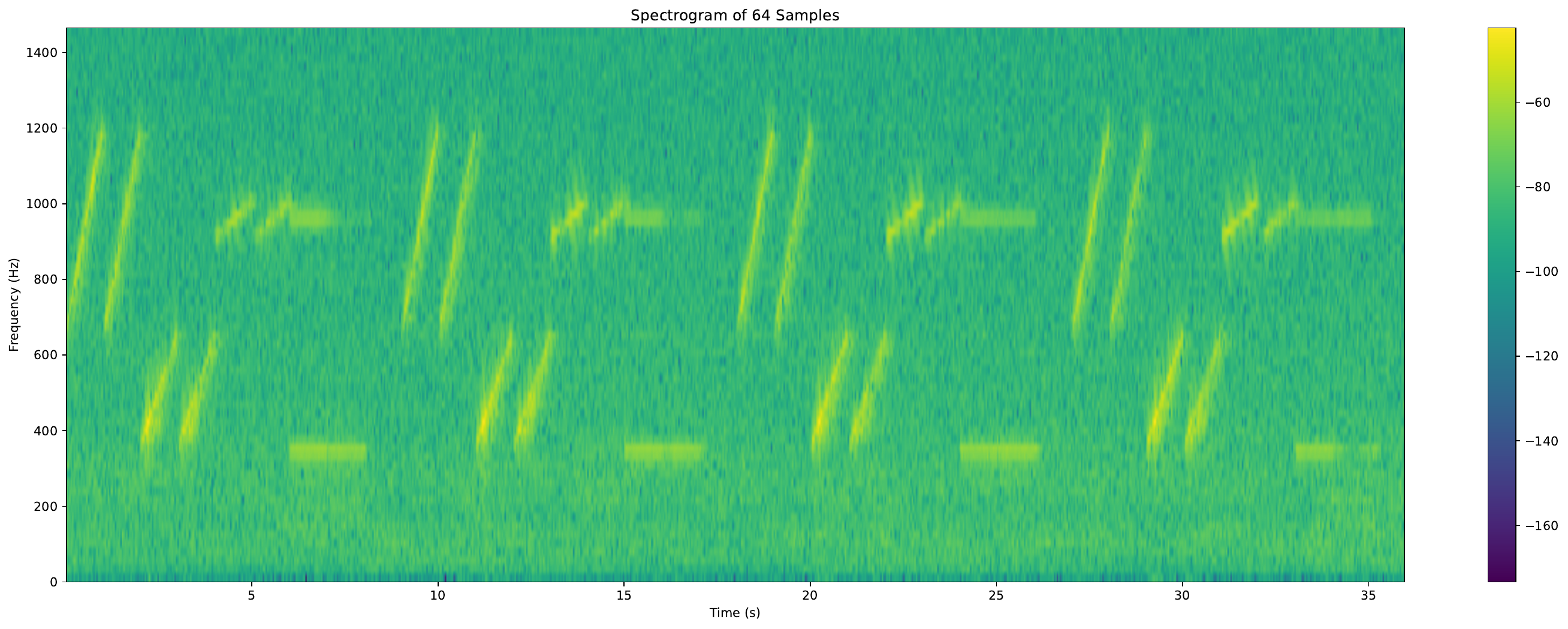}
\end{subfigure}
\begin{subfigure}[t]{0.45\textwidth}
    \includegraphics[width=1.1\linewidth]{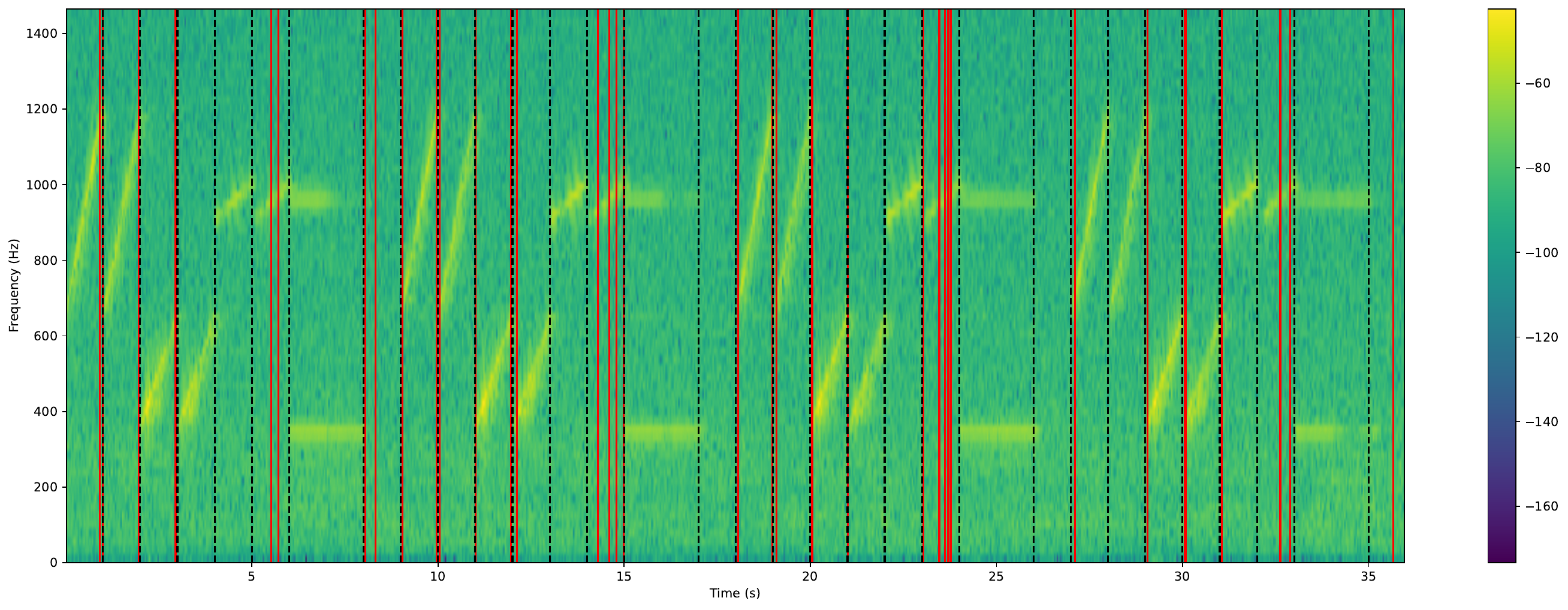}
\end{subfigure}
\caption{Underwater Acoustics Data annotated with detected change points. The parameters for this modified version of the algorithm are a windowsize of \(n=512\) for the short-time Fourier transform and \(q=0.96.\)}
\end{figure}\label{spectrogram}
\section{Discussion \& Conclusion}
\label{Discussion}
In this work, we have investigated the use of the Wasserstein metric as a tool for identifying change points in time series data and for clustering large data sets by lifting the problem into the space of distributions so that points can be considered as a part of a larger segment, thereby reducing the number of pairwise computations needed to form the similarity matrix. Our work extends previous considerations of the approach by removing parametric assumptions on the statistical structure of the data, and expanding the space of relevant data sets amenable to transport-based change point detection to include multi-dimensional data. We circumvent the curse of dimensionality by performing our change point detection and clustering on each component of the data set. 

In the future, this work could be extended by taking an approach to the clustering aspect of the problem which is more harmonious with the transport driven framework for the segmentation. In particular, one could leverage the geometry of Wasserstein barycenters to implement a \(K\)-means type algorithm, or potentially adapt a variation of density peaks-based clustering into the Wasserstein space. This approach would likely involve entropic regularization of the Wasserstein distance, due to the computational complexity involved in computing Wasserstein barycenters.  Also of interest is the optimization of the parameter choices \(w\) and \(q\). Elsewhere in the literature the problem of choosing an optimal window size for CPD algorithms based on the Kolmogorov-Smirnov test have been considered, and similar pursuits might be considered in the context of Wasserstein-based algorithms.

Furthermore, it is of interest to harmonize this approach with the Markov state modeling approach which is prominent in the computational chemistry community. 

\vspace{10pt}

\noindent\textbf{Acknowledgments:}  The authors gratefully acknowledge support from NSF DMS-2309519 and NSF DMS-2318894. We also extend our thanks to our colleagues at the Naval Research Laboratory Roger M. Oba and Laurie T. Fialkowski for helpful conversations and supply of the underwater acoustics dataset, as well as to Dr. Hannah J. Bates for her knowledge of chemical drawing software and willingness to help create figures.
\printbibliography

\end{document}